%% file: intersection.tex
\documentclass[reqno,11pt]{amsart}
 
\usepackage{amsmath}
\usepackage{amsthm}  
\usepackage{verbatim}
\usepackage{enumerate} 
\usepackage{mathtools}  
\usepackage{amssymb}
\usepackage{mathrsfs}
\usepackage[top=1.5in, bottom=1.5in, left=1in, right=1in]{geometry}  
\usepackage[colorlinks]{hyperref}
\hypersetup{
	colorlinks,
	citecolor=blue,
	linkcolor=red
}   
\numberwithin{equation}{section}
\usepackage{xypic,overpic}
\usepackage{bm}
\usepackage{tikz}
\usepackage{float}
\usepackage{cleveref}

\newtheorem{theorem}{\textbf{Theorem}}[section]
\newtheorem{theorem*}{\textbf{Theorem}}

\newtheorem{definition}[theorem]{\textbf{Definition}}
\newtheorem{proposition}[theorem]{\textbf{Proposition}}
\newtheorem{lemma}[theorem]{\textbf{Lemma}}

\newtheorem{corollary}[theorem]{\textbf{Corollary}}
\newtheorem{remark}[theorem]{\textbf{Remark}}

\newtheorem{example}[theorem]{\textbf{Example}}

\newtheorem{definition/proposition}[theorem]{\textbf{Definition/Proposition}}
\newtheorem{innercustomgeneric}{\customgenericname}
\providecommand{\customgenericname}{}
\newcommand{\newcustomtheorem}[2]{%
	\newenvironment{#1}[1]
	{%
		\renewcommand\customgenericname{#2}%
		\renewcommand\theinnercustomgeneric{##1}%
		\innercustomgeneric
	}
	{\endinnercustomgeneric}
}
\newcustomtheorem{customthm}{Theorem}
\newcustomtheorem{customprop}{Proposition}
\newcustomtheorem{customcoro}{Corollary}

\def\N{{\mathbb N}}
\def\R{\mathbb{R}}
\def\Z{{\mathbb Z}}

\def\C{{\mathbb C}}
\def\D{{\mathbb D}}

\def\Q{{\mathbb Q}}

\newcommand{\CP}{\mathbb{C}\mathbb{P}}

\def\cA{{\mathcal A}}
\def\cB{{\mathcal B}}

\def\cM{{\mathcal M}}
\def\cN{{\mathcal N}}

\def\cU{{\mathcal U}}

\def\rD{{\rm D}}

\def\rd{{\rm d}}

\def\la{\langle\,}
\def\ra{\,\rangle}

\def\std{\rm std}

\def\CZ{\rm CZ}

\newcommand{\eqdef}{=\joinrel=}

\DeclareMathOperator{\Ima}{im}
\DeclareMathOperator{\ind}{ind}

\DeclareMathOperator{\Id}{id}

\DeclareMathOperator{\sign}{sign}

\DeclareMathOperator{\End}{End}

\DeclareMathOperator{\virdim}{virdim}

\newcommand{\Addresses}{{
		\bigskip
		\footnotesize

	    Zhengyi Zhou, \par\nopagebreak
	    \textsc{Morningside Center of Mathematics, Chinese Academy of Sciences;}\par\nopagebreak
         \textsc{Academy of Mathematics and Systems Science, Chinese Academy of Sciences, China}\par\nopagebreak
		\textit{E-mail address}: \href{mailto:zhyzhou@amss.ac.cn}{zhyzhou@amss.ac.cn}

}}

\title{On the intersection form of fillings}
\author{Zhengyi Zhou}

\begin{document}
	\maketitle
\begin{abstract}
We prove, by an ad hoc method, that exact fillings with vanishing rational first Chern class of flexibly fillable contact manifolds have unique integral intersection forms. We appeal to the special Reeb dynamics (stronger than ADC in \cite{lazarev2016contact}) on the contact boundary, while a more systematic approach working for general ADC manifolds is developed independently by Eliashberg, Ganatra and Lazarev \cite{EGL}.  We also discuss cases where the  vanishing rational first Chern class assumption can be removed. We derive the uniqueness of diffeomorphism types of exact fillings of certain flexibly fillable contact manifolds and obstructions to contact embeddings, which are not necessarily exact. 

\end{abstract}
\input{s1}
\input{s2}
\input{s3}
\input{s4}
\input{s5}

\bibliographystyle{plain} 
\bibliography{ref}
\Addresses
\end{document}

%% file: s1.tex
\section{Introduction}
In his seminal work of introducing pseudo-holomorphic curves into symplectic geometry, Gromov \cite{MR809718} proved the uniqueness of exact fillings of $(S^3,\xi_{\std})$. In higher dimensions, the celebrated Eliashberg-Floer-McDuff theorem \cite{MR1091622} asserts the uniqueness of diffeomorphism type for exact fillings of $(S^{2n-1},\xi_{\std})$ whenever $n\ge 3$. Starting from those monumental results in the late 1980s - early 1990s, understanding the uniqueness of exact fillings of some contact manifolds has been a fundamental and influential question. In dimension $3$, the intersection theory of holomorphic curves can be used to construct foliations of symplectic fillings. A landmark result is Wendl's theorem on planar contact $3$-folds \cite{MR2605865}, which translates the classification of symplectic fillings into factorizations in mapping class groups. In higher dimensions, only ``homological" foliations by holomorphic curves can be obtained, just like the Eliashberg-Floer-McDuff theorem compared to Gromov's theorem. Based on ``homological" foliations, various generalizations of the Eliashberg-Floer-McDuff theorem were obtained, e.g.\ Oancea-Viterbo \cite{MR2874896} and Barth-Geiges-Zehmisch's \cite{MR4031531} works on topological uniqueness of exact fillings of subcritically fillable contact manifolds, Bowden-Gironella-Moreno \cite{BGM} and Geiges-Kwon-Zehmisch's  \cite{geiges2021diffeomorphism} works on topological uniqueness of exact fillings of the cosphere bundle $S^*T^n$. On the other hand, we studied the filling question from the perspective of Floer theories and obtained various uniqueness results including uniqueness on cohomology groups \cite{MR4182808,filling} or rings \cite{ring}, diffeomorphism types \cite{product,quotient} as well as certain properties of the symplectic cohomology \cite{filling,zhou2019symplectic}. Unlike the cohomology ring, which is a homotopy invariant, the intersection form $H_p(W^{2n};\Z)\otimes H_{2n-p}(W^{2n};\Z)\to \Z$ is more sensitive to the manifold structure. In this note, we show the uniqueness of the integral intersection form for exact fillings of some flexibly fillable contact manifolds, which shall yield the uniqueness of diffeomorphism types in some cases. 

\begin{theorem}\label{thm:main}
Assume $(Y^{2n-1},\xi)$ has a flexible filling $W^{flex}$ with vanishing rational first Chern class $c^{\Q}_1(\xi)$, then the integral intersection form of any exact filling $W$ with $c^\Q_1(W)=0$ is isomorphic to the integral intersection form of the flexible filling $W^{flex}$, that is there is an isomorphism between $H_n(W;\Z)$ and $H_n(W^{flex};\Z)$, such that the intersection form on them are identified. 
\end{theorem}
\begin{remark}
    By \Cref{prop:simple} below, we have $H^*(W^{flex};\Z)=H^*(W;\Z)$, which then implies that $H_*(W^{flex};\Z)=H_*(W;\Z)$ by the universal coefficient theorem. Therefore the only non-trivial information about the intersection form is contained in the middle dimension $H_n(W;\Z)\otimes H_n(W;\Z)\to \Z$. \Cref{prop:SS} below lists some cases where the $c^\Q_1(W)=0$ assumption can be dropped.
\end{remark}

The idea of the proof is as follows: take two cycles $A,B$ in $H_n(W;\Z)$, under the transversality assumption, the intersection number $A\cdot B = \# (A\pitchfork B)$, where the latter is a finite set of oriented points. By \cite{filling}, we have $SH^*(W;\Z)=0$ for any topologically simple (i.e.\ $c^{\Q}_1(W)=0$ and $H_1(Y;\Q) \to H_1(W;\Q)$ is injective, the comparison with topologically simple fillings used in \cite{MR4182808,filling} is in \S \ref{s3}.) exact filling. Geometrically, it means that there is one curve counted algebraically passing through a fixed point. If we choose the point constraint to be $A\pitchfork B$, which might consist of several points, this allows us to view $A\cdot B$ as the counting of one boundary component of a $1$-dimensional moduli space with one positive puncture asymptotic to a non-constant Hamiltonian orbit and two negative punctures asymptotic to points on $A$ and $B$ as removable singularities. Then we can establish the independence result by looking at the other boundaries and applying neck-stretching. It is important that $Y$ is flexibly fillable rather than just being asymptotically dynamically convex (ADC) \cite[Definition 3.6]{lazarev2016contact}, as we need to use the special Reeb dynamics on the boundary of a flexibly fillable contact manifold constructed by Lazarev \cite{lazarev2016contact}. On the other hand, a more systematic approach is independently developed by Eliashberg, Ganatra, and Lazarev \cite{EGL} using the secondary coproduct defined by Ekholm and Oancea \cite{ekholm2017symplectic}, their argument is also sketched in \cite[Theorem 4.1]{quotient}. In particular, their results apply to general ADC contact manifolds. Our approach is ad hoc in the sense that we present the intersection number as a counting of holomorphic curves contained in the cylindrical end of the filling. It is the combination of such curves that carries the information, while some moduli spaces that appear in the process do not bear meanings as structural maps in symplectic cohomology or symplectic field theory. We discuss some conjectural relations between the two approaches at the end of \S \ref{s3}.

\subsection{Dropping the topologically simple assumption}
The topologically simple condition in \cite{MR4182808,filling}  was used to make sure the grading of symplectic cohomology of a filling is consistent with the grading induced from the contact boundary so that we can do effective dimension computation after neck-stretching. We first note that the $H_1(Y;\Q)\to H_1(W;\Q)$ injective condition can be dropped for flexibly fillable contact manifolds (but not for general ADC manifolds).

\begin{proposition}\label{prop:simple}
	Let $(Y^{2n-1},\xi)$ be a contact manifold filled by a flexible Weinstein domain $W^{flex}$ with $c^{\Q}_1(\xi)=0$. For any exact filling $W$ of $Y$ with $c_1^\Q(W)=0$, we have that $H^*(W;\Z)\to H^*(Y;\Z)$ is isomorphic to  $H^*(W^{flex};\Z)\to H^*(Y;\Z)$.
\end{proposition}

Following a different perspective which was used in \cite{quotient}, by exploiting the degeneracy of a spectral sequence from the ``boundary" grading, we can drop the $c_1^\Q(W)=0$ condition in some cases.
\begin{theorem}\label{prop:SS}
	Let $(Y^{2n-1},\xi)$ be a contact manifold filled by a flexible Weinstein domain $W^{flex}$ with $c^{\Q}_1(\xi)=0$.
	\begin{enumerate}
	    \item\label{c1} If  $H^n(W^{flex};\Q)\to H^n(Y;\Q)$ is injective for the flexible filling $W^{flex}$,  then for any exact filling $W$ of $Y$, we have $H^*(W;\Z)\to H^*(Y;\Z)$ is isomorphic to $H^*(W^{flex};\Z)\to H^*(Y;\Z)$. In particular, we have $c_1^\Q(W)=0$, as $H^2(W^{flex};\Z)\to H^2(Y;\Z)$ is injective and $c_1^{\Q}(\xi)=0\in H^2(Y;\Q)$.
	    \item\label{c2} If $n$ is even, then for any exact filling $W$ of $Y$, we have $SH^*(W;\Z)=0$ and $\dim \oplus_{*=0}^{2n}H^*(W;\Q)\le \dim \oplus_{*=0}^{2n}H^*(W^{flex};\Q)$. Moreover, 
	    \begin{enumerate}
	        \item\label{a} if $H^1(W^{flex};\Q)=0$, then $\dim \oplus_{*=0}^{2n}H^*(W;\Q)=\dim \oplus_{*=0}^{2n}H^*(W^{flex};\Q)$;
	        \item\label{b} if moreover $\dim \ker (H^n(W^{flex};\Q)\to H^n(Y;\Q))=1$, then $H^*(W;\Z)\to H^*(Y;\Z)$ is isomorphic to $H^*(W^{flex};\Z)\to H^*(Y;\Z)$. In particular, we have $c_1^\Q(W)=0$ as before.
	    \end{enumerate}
	\end{enumerate}
\end{theorem}
For example, the flexible version $\mathrm{Flex}(T^*S^n)$ of $T^*S^n$ for odd $n$ satisfies the condition \eqref{c1}, and it satisfies condition \eqref{b} when $n$ is even. Therefore any exact filling of $\partial \mathrm{Flex}(T^*S^n)$ is topologically simple using \Cref{prop:simple} and \Cref{prop:SS}. However, in the $H^n(W^{flex};\Q)\to H^n(Y;\Q)$ injective case, the intersection form of $W^{flex}$ is necessarily trivial. To see this, since the intersection form $H_n(W^{flex};\Z)\otimes H_n(W^{flex};\Z) \to H_0(W^{flex};\Z)=\Z$ is isomorphic to the cup product $H^n(W^{flex},Y;\Z)\otimes H^n(W^{flex},Y;\Z)\to H^{2n}(W^{flex},Y;\Z)$ using the Lefschetz duality. If $H^n(W^{flex};\Q)\to H^n(Y;\Q)$ injective, then $H^n(W^{flex},Y;\Q)\to H^n(W^{flex};\Q)$ is trivial, hence $H^n(W^{flex},Y;\Q)\otimes H^n(W^{flex},Y;\Q) \stackrel{\cup}{\to}  H^n(W^{flex};\Q)\otimes H^n(W^{flex},Y;\Q)\stackrel{\cup}{\to}  H^{2n}(W^{flex},Y;\Q)=\Q$ is trivial, which implies the intersection form is trivial. By rank of the intersection form, we mean the rank of the bilinear form on $H_n(W^{flex};\Q)$. Since $H^n(W^{flex};\Q)\otimes H^n(W^{flex},Y;\Q)\stackrel{\cup}{\to}  H^{2n}(W^{flex},Y;\Q)=\Q$ is non-degenerate, the rank is the dimension of the image of $H^n(W^{flex},Y;\Q)\to H^n(W^{flex};\Q)$.

\subsection{Uniqueness of diffeomorphism types}
Using \Cref{prop:SS}, we can get the following uniqueness result of diffeomorphism types for symplectic fillings following the same topological argument in \cite{filling,product}.
\begin{theorem}\label{thm:flexQ}
Let $Q$ be a closed manifold such that $\pi_1(Q)$ is abelian and $\chi(Q)=0$, then the interior of any exact filling of the contact boundary of $\mathrm{Flex}(T^*Q)$ (the flexible version of $T^*Q$) is diffeomorphic to $T^*Q$ as an open manifold.
\end{theorem}

For highly connected manifolds, the diffeomorphism type is restricted to a large extent by its intersection form, c.f.\ \cite{smale1962structure}. In view of this, we have the following corollary of \Cref{thm:main}.
\begin{theorem}\label{cor:diff}
If $n \equiv 6 \mod 8$ and $2n$-dimensional flexible domain $W^{flex}$ only has flexible $n$-handles except for the $0$-handle such that the intersection form is even, then any exact filling $W$ of $\partial W^{flex}$ with $c^{\Q}_1(W)=0$ is diffeomorphic to $W^{flex}$. If moreover the intersection form $W^{flex}$ is of rank $1$,  then any exact filling  $W$ of $\partial W^{flex}$ is diffeomorphic to $W^{flex}$.
\end{theorem}

It is worth noting that the uniqueness of diffeomorphism type above follows from completely different topological argument compared to results in \cite{MR4031531,BGM,geiges2021diffeomorphism,MR1091622,product} as well as \Cref{thm:flexQ}, where the uniqueness hinges on $H^*(W;\Z)\to H^*(Y;\Z)$ being injective and the $h$-cobordism theorem, which is rarely the case in \Cref{cor:diff}. Moreover, Wall \cite{MR145540} proved that the diffeomorphism type of a smooth, closed, oriented, $(n-1)$-connected $2n$-manifold of dimension at least $6$ is determined, up to connected sum with a homotopy sphere, by the middle homology group, the intersection pairing, and the so-called normal bundle data. It is an interesting question to study whether such normal bundle data can be approached symplectically in the flexible domain case. 

\subsection{Obstructions to contact embeddings}
Let $V$ be a symplectic manifold, we call a closed hypersurface $Y\subset V$ a contact hypersurface if there exists a local Liouville vector field near $Y$ that is transverse to $Y$. When $V$ is exact, and the Liouville vector field can be defined globally, we say $Y$ is an exact contact hypersurface. When $Y$ is separating,  which is automatic if $H_{2n-1}(V;\Z)=0$, and is an exact contact hypersurface, the Liouville vector field must point out along $Y$ w.r.t.\ the compact domain (inside the completion of $V$) bounded by $Y$. Note that if we drop the exactness assumption, the local Liouville vector field can either point out or point in along $Y$ w.r.t.\ the compact domain it bounds. Understanding whether a contact manifold can be embedded in a symplectic manifold, e.g.\ $\C^n$, is a fundamental question. In dimension $4$, it is closely related to the Gompf conjecture \cite{MR3145144}, see \cite{MR4392461} for recent advances on this conjecture.  Obstructions to exact contact hypersurfaces were studied by Cieliebak and Frauenfelder \cite{MR2461235} using the Rabinowitz-Floer homology. For example, they showed that there is no exact embedding of $S^*S^n$ into $\C^n$.  Using \Cref{thm:main}, we have the following obstructions to contact embeddings (not necessarily exact), where the obstructions from the Rabinowitz-Floer homology typically vanish. 

\begin{theorem}\label{thm:emb}
Let $V$ be a $2n$-dimensional exact domain with $c^\Q_1(V)=0$ and $W$ a flexible Weinstein domain. Suppose the rank of the intersection form on $H_n(V;\Q)$ is smaller than that of $W$.  Then $\partial W$ cannot be embedded into $V$ as a separating contact hypersurface with the local Liouville vector field points out w.r.t.\ the bounded domain. If we assume moreover that $V$ is $P\times \D$ for an exact domain $P$,  $\partial W$ can not be embedded into $V$ as a contact hypersurface. 
\end{theorem}

For example, we can consider the flexible version $W^{flex}$ of a Brieskorn variety $\{x_0^{a_0}+\ldots+x_n^{a_n}=1\}\subset \C^{n+1}$ for $a_i\in \N_+$. When $a_i\ge 2$ for all $i$, the intersection form of the Brieskorn variety/$W^{flex}$ has a positive rank if one of $a_i\ge 3$ or $n$ is even. Then by \Cref{thm:emb}, there is no contact embedding of the contact boundary $\partial W^{flex}$ into $\C^n$. Note that for suitable choices of $a_i$, the contact boundary $\partial W^{flex}$ has the same almost contact structure as $(S^{2n-1},\xi_{\std})$ \cite{MR3483060}, in particular, such embedding is not obstructed topologically.


\subsection*{Organization of the paper}
We recall the basics of flexible Weinstein domains, symplectic cohomology, and neck-stretching in \S \ref{s2}. In \S \ref{s3}, we compare different notions of ADC and topologically simple fillings and discuss situations where the topologically simple assumption holds automatically. We prove \Cref{thm:main} in \S \ref{s4} and applications in \S \ref{s5}.

\subsection*{Acknowledgments}
The author is grateful to Sheel Ganatra for sharing the draft of \cite{EGL} and helpful conversions, to Jonathan Bowden and Oleg Lazarev for enlightening discussions, and to the anonymous referee for many suggestions. The author is supported by the National Key R\&D Program of China under Grant No.\ 2023YFA1010500, the National Natural Science Foundation of China under Grant No.\ 12288201 and 12231010.

%% file: s2.tex
\section{Preliminaries on flexibly fillable manifolds and symplectic cohomology}\label{s2}
\subsection{Preliminaries on flexibly fillable contact manifolds}\label{ss21}
We assume throughout this note that $(Y,\xi)$ is a contact manifold of dimension $2n-1\ge 5$, such that $\xi$ is co-oriented and its first rational Chern class $c^{\Q}_1(\xi)$ is zero.
\begin{definition} We first recall definitions of symplectic fillings of various flavors:
	\begin{itemize}
		\item $(W,\lambda)$ is a Liouville filling of $(Y,\xi)$ iff $\rd \lambda$ is a symplectic form on $W$, the Liouville field $X_\lambda$, defined by $i_{X_\lambda} \rd \lambda=\lambda$, is outward transverse along $\partial W$, and $(\partial W, \ker \lambda|_{\partial W})$ is contactomorphic to $(Y,\xi)$.  
		\item $(W,\lambda,\phi)$ is a Weinstein filling of $(Y,\xi)$ iff $(W,\lambda)$ is a Liouville filling, $\phi:W\rightarrow \mathbb{R}$ is a Morse function with maximal level set $\partial W$, and $X_\lambda$ is a gradient-like vector field for $\phi$.
		\item $(W,\lambda,\phi)$ is a flexible Weinstein filling of $(Y,\xi)$ iff $(W,\lambda,\phi)$ is a Weinstein filling, and there exist regular values  $c_1<\min \phi<c_2<\ldots<c_k<\max \phi$  of $\phi$, such that there are no critical values in $[c_k,\max \phi]$ and each $\phi^{-1}([c_i,c_{i+1}])$ is a Weinstein cobordism with a single critical point $p$ whose attaching sphere $\Lambda_p$ is either subcritical or a loose Legendrian in $(\phi^{-1}(c_i), \lambda|_{\phi^{-1}(c_i)})$, see \cite[Definition 11.29]{MR3012475} and \cite[Definition 2.4]{lazarev2016contact}.
	\end{itemize}
\end{definition}

Let $(Y,\xi)$ be a contact manifold such that the rational first Chern class $c_1^{\Q}(\xi)\in H^2(Y;\Q)$ is zero. Therefore there exists $N\in \N_+$, such that $Nc_1(\xi)=0\in H^2(Y;\Z)$. Since $c_1(\det_{\C} \oplus^N\xi)=Nc_1(\xi)=0$, $\det_{\C} \oplus^N\xi$ can be trivialized, where we view $\xi$ as a complex bundle using a complex structure compatible with $\rd \alpha|_{\xi}$. Then we can assign a rational Conley-Zehnder index $\mu_{\CZ}$ to each non-degenerate Reeb orbit $\gamma$, i.e.\ $1/N$ of the Conley-Zehnder index of $\oplus^N \Phi(t)$ under a symplectic trivialization of $\oplus^N \gamma^*\xi$ inducing the fixed trivialization on $\det_{\C} \oplus^N\xi$, where $\Phi(t)$ is the linearization of the Reeb flow along $\gamma$ restricted to $\xi$, see \cite{gironella2021exact,Reeb} for details. For those orbits with torsion homology classes, the Conley-Zehnder index is independent of $N$ and the trivialization \cite[Proposition 3.8]{gironella2021exact}. For contractible Reeb orbits, it is the usual integer-valued Conley-Zehnder index using any induced trivialization from a bounding disk.

In \cite{lazarev2016contact}, Lazarev introduced the concept of asymptotically dynamically convex (ADC) manifolds, and in the process of proving flexible fillable contact manifolds are ADC, he proved the following property:

\begin{proposition}\label{prop:reeb}
	Let $(Y^{2n-1},\xi)$ be a flexibly fillable contact manifold with $c^{\Q}_1(\xi)=0$ and a fixed contact form $\alpha_0$. Then for any $D \gg 0$, there exists a contact form $\alpha_D<\alpha_0$ such that Reeb orbits of $\alpha$ with period smaller than $D$ are non-degenerate and have Conley-Zehnder index $\ge 1$ (for any fixed trivialization of $\det_{\C} \oplus^N\xi$ ). And those orbits with Conley-Zehnder index $1$ are simple.
\end{proposition}
\begin{proof}
	This follows from the proof of \cite[Theorem 3.15, 3.17, 3.18]{lazarev2016contact}. For $D \gg 0$ and a suitable $\alpha<\alpha_0$, if we consider Reeb orbits of period smaller than $D$, then they fall into the following two classes: (1) Each subcritical handle of index $k$ creates a simple contractible Reeb orbit with Conley-Zehnder index $n+1-k$, and all multiple covers of it have higher Conley-Zehnder indices; (2) Every loose handle attachment creates (several) contractible simple Reeb orbits of Conley-Zehnder index $1$ and many other orbits with Conley-Zehnder index strictly greater than $1$. 
\end{proof}

\subsection{Symplectic cohomology}\label{s22}
We will only recall the basics of symplectic cohomology to set up notations and relevant structures for our main results. We refer readers to \cite{cieliebak2018symplectic,ritter2013topological,seidel2006biased} for a more complete treatment of the subject.

By a strict contact manifold, we mean a contact manifold $Y$ equipped with a contact form $\alpha$ and an exact filling $(W,\lambda)$ is a strict exact filling $(Y,\alpha)$ if $\alpha=\lambda|_Y$. Given an exact filling $(W,\lambda)$, we consider the completion $(\widehat{W},\widehat{\lambda})$, where $\widehat{W}=W\cup (1,\infty)_r\times \partial W$, $\widehat{\lambda}$ is $\lambda$ on $W$ and $r(\lambda|_{\partial W})$ on $(1,\infty)_r\times \partial W$. Roughly speaking, symplectic cohomology is the ``Morse cohomology" of the free loop space w.r.t.\ the symplectic action functional
\begin{equation}\label{eqn:action}
    \cA_H(x):=-\int x^*\widehat{\lambda} +\int_{S^1} (x^*{H})\rd t
\end{equation}
where $H:S^1\times \widehat{W}\to \R$ is a Hamiltonian on the completion such that ``the slope $\frac{\partial }{\partial r}H$" goes to infinity as $r$ goes to infinity. Let $R$ be any commutative ring, the cochain complex $CF^*(H)$ is a free $R$-module generated by $1$-periodic orbits of $X_H$ when $H$ is non-degenerate. 
Let $x,y$ be two generators in $CF^*(H)$, we use $\cM^{x}_{y}$ to denote the compactified moduli space of Floer cylinders from $x$ to $y$:
$$\cM^x_{y}=\overline{\left\{u:\R_s\times S^1_t\to \widehat{W}\left|\partial_su+J(\partial_tu-X_H)=0, \displaystyle \lim_{s\to \infty} u = x, \lim_{s\to-\infty} u=y\right. \right\}/\R }.$$

We say an almost complex structure $J$ on $\widehat{W}$ is admissible, if $J$ is compatible with $\rd \widehat{\lambda}$ and $J$ is cylindrically convex for $r\ge B \gg 0$ so that the integrated maximum principle \cite{MR2602848,cieliebak2018symplectic}) holds, where by cylindrically convex, we mean $J$ preserves $\xi$ and $J(r\partial_r)=R$ on $[B,\infty)_r\times \partial W$ with $\xi=\ker (\lambda|_{\partial W})$ and $R$ the Reeb vector field of $(\partial W , \lambda|_{\partial W})$. In both non-degenerate and Morse-Bott cases,  for a generic $S^1$-dependent admissible complex structure $J:S^1\to \End(T^*\widehat{W})$, $\cM^{x}_{y}$ is cut out transversely as a manifold with boundary for those with expected dimension $\virdim \cM_{x,y}\le 1$. Moreover, $\cM^{x}_{y}$ can be oriented in a coherent way such that 
$$\delta(x)=\sum_{y,\virdim \cM^{x}_{y}=0}(\# \cM^{x}_{y}) y$$
defines a differential on $CF^*(H)$.

The symplectic cochain complex is graded by $n-\mu_{\CZ}(x)$. In general, since $\mu_{\CZ}$ is only well-defined in $\Z/2$, symplectic cohomology always has a $\Z/2$ grading. If $c_1(W)=0$, upon fixing a trivialization of $\det_{\C} TW$, $\mu_{\CZ}$ is well-defined in $\Z$, which is independent of the trivialization for any periodic orbit with finite order homology class. Moreover, if $c^{\Q}_1(W)=0$, then  $\mu_{\CZ}$ is well-defined in $\Q$ (as in \S \ref{ss21}) if we fix a trivialization of $\det_{\C}\oplus^N TW$ and symplectic cohomology can be graded by $\Q$. We say an exact filling $W$ of $Y$ is \emph{topologically simple} if $H_1(Y;\Q)\to H_1(W;\Q)$ is injective and $c_1^{\Q}(W)=0$. 

\begin{proposition}\label{prop:trivialization}
    Assume $W$ is topologically simple, then there exists $N,M\in \N_+$, such that $\det_{\C}\oplus^N\xi$ can be trivialized, and the induced trivialization on $\det_{\C}\oplus^{NM}\xi$ from any trivialization of $\det_{\C}\oplus^N \xi$ can be extended to a trivialization of $\det_{\C} \oplus^{NM}TW$.
\end{proposition}
\begin{proof}
    Since $c_1^{\Q}(W)=0$, we may assume $Nc_1(W)=0\in H^2(W;\Z)$ for $N\in \N_+$. Therefore we have a trivialization of $\det_{\C} \oplus^N TW$, which induces an trivialization of $\det_{\C}\oplus^N\xi$. Now, different trivializations of $\det_{\C}\oplus^N\xi$ are parameterized the gauge group $[Y:S^1]=H^1(Y;\Z)$. Since $H_1(Y;\Q)\to H_1(W;\Q)$ is injective, we have $H^1(W;\Q)\to H^1(Y;\Q)$ is surjective and hence $H^1(W;\Z)\to H^1(Y;\Z)$ is surjective onto $M\cdot H^1(Y;\Z)\subset H^1(Y;\Z)$ for some $M\in \N_+$. As the difference between two induced trivializations on $\det_{\C}\oplus^{NM}\xi$ from two trivializations of $\det_{\C}\oplus^N \xi$ is given by en element in $M\cdot H^1(Y;\Z)$, the claim follows.
\end{proof}

\subsection{Cascades model for the Morse-Bott case}\label{ss23}
It is often convenient to consider symplectic cohomology using Morse-Bott Hamiltonians. In this paper, we will consider the following two special Morse-Bott scenarios:
\begin{enumerate}
    \item $H=0$ on $W$\footnote{Strictly speaking, when $H=0$ on $W$, it is not Morse-Bott non-degenerate along the boundary $W$. However, this failure of Morse-Bott non-degeneracy does not matter in the cascades construction by \cite[Proposition 2.6]{zhou2019symplectic}.};
    \item $H=h(r)$ on  $ (1,\infty)_r \times \partial W $ with $h'(r)=D$ for $r\gg 0$ and $h''(r)>0$ unless $h=0$ or $h'(r)=D$. 
\end{enumerate}
The purpose of introducing the cacades model for the Morse-Bott case is twofold: (1) when $H=0$ on $W$, it is cleaner to apply neck-stretching; (2) when the $H$ is autonomous, the $1$-periodic Hamiltonian orbits form $S^1$ Morse-Bott family and we can use non-negativity of the contact energy in \Cref{prop:SFT2} to constrain Floer cylinders.

We choose a Morse function $g$ on $W$, such that $\partial_r g>0$ on $\partial W$.  We use $\overline{\gamma}$ to denote the $S^1$ family of Hamiltonian orbits corresponding to the Reeb orbit $\gamma$. Then we pick two different generic points $\hat{\gamma}$ and $\check{\gamma}$ on $\Ima \overline{\gamma}$, this is equivalent to choosing a Morse function $g_{\overline{\gamma}}$ with one maximum and one minimum on $\Ima \overline{\gamma}$ in \cite[\S 3]{bourgeois2009symplectic}. By \cite[Lemma 3.4]{bourgeois2009symplectic}, the Morse function $g_{\overline{\gamma}}$ can be used to perturb the Hamiltonian $H$ to get two non-degenerate orbits from $\overline{\gamma}$, which are often denoted by $\hat{\gamma}$ and $\check{\gamma}$ in literatures with $\mu_{\CZ}(\hat{\gamma})=\mu_{\CZ}(\gamma)+1$ and $\mu_{\CZ}(\check{\gamma})=\mu_{\CZ}(\gamma)$.   

Then we have a Floer cochain complex $CF^*(H)$, which is a free $R$-module generated by critical points of $g$ with Morse index as grading, and two generators $\hat{\gamma},\check{\gamma}$ for each Reeb orbit $\gamma$ of period smaller than $D$, with gradings $n-\mu_{\CZ}(\gamma)-1$ and $n-\mu_{\CZ}(\gamma)$ (interpreted as in $\Z/2$,  $\Z$ or $\Q$ depending on the situation). The differential is defined by counting rigid cascades \cite[(39),(40)]{bourgeois2009symplectic}. 
In this case, $\cM^{x}_{y}$ is the compactified moduli space of \emph{cascades} from $x$ to $y$, which can be pictorially described as 
    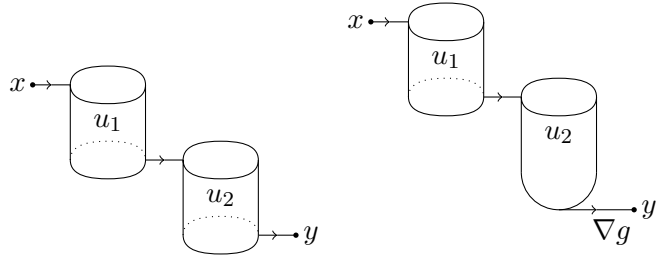
\begin{figure}[H]\label{fig:cascades}
	\begin{tikzpicture}
	\draw (0,0) to [out=90, in = 180] (0.5, 0.25) to [out=0, in=90] (1,0) to [out=270, in=0] (0.5,-0.25)
	to [out = 180, in=270] (0,0) to (0,-1);
	\draw[dotted] (0,-1) to  [out=90, in = 180] (0.5, -0.75) to [out=0, in=90] (1,-1);
	\draw (1,-1) to [out=270, in=0] (0.5,-1.25) to [out = 180, in=270] (0,-1);
	\draw (1,0) to (1,-1);
	\draw[->] (1,-1) to (1.25,-1);
	\draw (1.25,-1) to (1.5,-1);
	\draw (1.5,-1) to [out=90, in = 180] (2, -0.75) to [out=0, in=90] (2.5,-1) to [out=270, in=0] (2,-1.25)
	to [out = 180, in=270] (1.5,-1) to (1.5,-2);
	\draw[dotted] (1.5,-2) to  [out=90, in = 180] (2, -1.75) to [out=0, in=90] (2.5,-2);
	\draw (2.5,-2) to [out=270, in=0] (2,-2.25) to [out = 180, in=270] (1.5,-2);
	\draw (2.5,-1) to (2.5,-2);
	\draw[->] (2.5,-2) to (2.75,-2);
	\draw (2.75,-2) to (3,-2);
	\draw[->] (-0.5,0) to (-0.25,0);
	\draw (-0.25,0) to (0,0);
	\node at (0.5,-0.5) {$u_1$};
	\node at (2, -1.5) {$u_2$};
	\fill (-0.5,0) circle[radius=1pt];
	\node at (-0.7, 0) {$x$};
	\fill (3,-2) circle[radius=1pt];
	\node at (3.2, -2) {$y$};
	\end{tikzpicture}
	\begin{tikzpicture}
	\fill (-0.5,0) circle[radius=1pt];
	\node at (-0.7, 0) {$x$};
	\draw (0,0) to [out=90, in = 180] (0.5, 0.25) to [out=0, in=90] (1,0) to [out=270, in=0] (0.5,-0.25)
	to [out = 180, in=270] (0,0) to (0,-1);
	\draw[dotted] (0,-1) to  [out=90, in = 180] (0.5, -0.75) to [out=0, in=90] (1,-1);
	\draw (1,-1) to [out=270, in=0] (0.5,-1.25) to [out = 180, in=270] (0,-1);
	\draw (1,0) to (1,-1);
	\draw[->] (1,-1) to (1.25,-1);
	\draw (1.25,-1) to (1.5,-1);
	\draw (1.5,-1) to [out=90, in = 180] (2, -0.75) to [out=0, in=90] (2.5,-1) to [out=270, in=0] (2,-1.25)
	to [out = 180, in=270] (1.5,-1) to (1.5,-2);
	\draw (2.5,-2) to [out=270, in=0] (2,-2.5) to [out = 180, in=270] (1.5,-2);
	\draw (2.5,-1) to (2.5,-2);
	\draw[->] (-0.5,0) to (-0.25,0);
	\draw (-0.25,0) to (0,0);
	\draw[->] (2,-2.5) to (2.5,-2.5);
	\draw (2.5,-2.5) to (3,-2.5);
	\fill (3,-2.5) circle[radius=1pt];
	\node at (3.2,-2.5) {$y$};
    \node at (2.7,-2.8) {$\nabla g$};
	\node at (0.5,-0.5) {$u_1$};
	\node at (2, -1.5) {$u_2$};
	\end{tikzpicture}
	\caption{$2$ level cascades}\label{fig:cascades}
\end{figure}
\begin{enumerate}
	\item Each unlabeled horizontal arrow is a \emph{negative} gradient flow of $g_{\overline{\gamma}}$ in $\Ima \overline{\gamma}$, i.e.\ flow toward $\check{\gamma}$.
    \item The labeled horizontal arrow is a \emph{positive} gradient flow of $\nabla g$ using a generic metric.
	\item $u$ is a solution to the Floer equation $\partial_s u+J_t(\partial_tu-X_H)=0$ modulo $\R$ translation.
	\item Every intersection point of the line with the surface satisfies the obvious matching condition.
\end{enumerate}
More formally, the differential is defined by counting the following compactified moduli spaces.
\begin{enumerate}
	\item For $p,q\in Crit(g)$, 
	$$\cM^{p}_{q}:=\overline{\{ \gamma:\R_s \to W|\gamma'+\nabla g =0, \lim_{s\to \infty} \gamma=p,\lim_{s\to -\infty} \gamma=q \}/\R}.$$
	\item For $\gamma_+,\gamma_- \in \{\check{\gamma},\hat{\gamma}| \forall S^1 \text{ family of orbits } \overline{\gamma}\}$, a $k$-cascade from $\gamma_+$ to $\gamma_-$ is a tuple $(u_1,l_1,\ldots,l_{k-1},u_k)$, such that
	\begin{enumerate}
		\item $l_i$ are positive real numbers.
		\item nontrivial $u_i\in \{u:\R_s\times S^1_t\to \widehat{W}| \partial_su+J_t(\partial_tu-X_H)=0, \displaystyle \lim_{s\to \infty} u \in \overline{\gamma}_{i-1}, \displaystyle \lim_{s\to -\infty} u \in \overline{\gamma}_{i}\}/\R$ such that $\gamma_+\in \overline{\gamma}_0$ and $\gamma_-\in \overline{\gamma}_{k}$, where the $\R$ action is the translation on $s$.
		\item $\phi_{-\nabla g_{\overline{\gamma}_i}}^{l_i}(\displaystyle\lim_{s\to -\infty}u_i(s,0))=\lim_{s\to \infty}u_{i+1}(s,0)$ for $1\le i\le k-1$, $\gamma_+=\displaystyle\lim_{t\to \infty}\phi^{-t}_{-\nabla g_{\overline{\gamma}_0}}(\lim_{s\to \infty} u_1(s,0))$, and $\gamma_-=\displaystyle\lim_{t\to \infty}\phi^{t}_{-\nabla g_{\overline{\gamma}_k}}(\lim_{s\to -\infty} u_k(s,0))$, where $\phi^t_{-\nabla g_{\overline{\gamma}}}$ is the time $t$ flow of $-\nabla g_{\overline{\gamma}}$ on $\Ima \overline{\gamma}$.
	\end{enumerate}
	Then we define $\cM^{\gamma_+}_{\gamma_-}$ to be the compactification of the space of all cascades from $\gamma_+$ to $\gamma_-$. The compactification involves the usual Hamiltonian-Floer breaking of $u_i$ as well as degeneration corresponding to $l_i=0,\infty$.  The $l_i=0$ degeneration is equivalent to a Hamiltonian-Floer breaking $\displaystyle\lim_{s\to -\infty}u_i=\lim_{s\to \infty}u_{i+1}$. In particular, they can be glued or paired, hence do not contribute (algebraically) to the boundary of $\cM^{\gamma_+}_{\gamma_-}$. The $l_i=\infty$ degeneration is equivalent to a Morse breaking for $g_{\overline{\gamma}_i}$, which will contribute to the boundary of $\cM^{\gamma_+}_{\gamma_-}$.

	\item For $\gamma_+ \in  \{\check{\gamma},\hat{\gamma}| \forall S^1 \text{ family of orbits } \overline{\gamma}\}$ and $q\in Crit(g)$, a $k$-cascades from $\gamma_+$ to $q$ is a tuple  $(l_0,u_1,l_1,\ldots,u_k,l_{k})$ as before, except
	\begin{enumerate}
		\item $u_k\in \{u:\C \to \widehat{W}|\partial_s u+J_t(\partial_tu-X_H)=0,\displaystyle\lim_{s\to \infty} u \in \overline{\gamma}_{k-1}, u(0)\in W^{\circ}\}/\R$, where we use the identification $\R\times S^1 \to \C^*, (s,t)\mapsto e^{2\pi(s+it)}$. $W^{\circ}$ is the interior of $W$, where the Floer equation is $\overline{\partial}_J u=0$, hence the removal of singularity implies that $u(0)$ is a well-defined notation.
		\item $q=\displaystyle\lim_{t\to \infty}\phi^t_{\nabla g}(u_k(0))$.
	\end{enumerate}
    Then $\cM^{\gamma_+}_{q}$ is defined to be the compactification of the space of all cascades from $\gamma_+$ to $q$. 
\end{enumerate}
For generic choices of almost complex structures, $\cM^{x}_{y}$ is cut out transversely whenever the virtual dimension is at most $1$ and $\cM^{x}_{y}$ can be oriented in a coherent way such that 
$$\delta(x)=\sum_{y,\virdim \cM^{x}_{y}=0}(\# \cM^{x}_{y}) y$$
defines a differential on $CF^*(H)$ as before.

The symplectic cohomology $SH^*(W;\R)$ is defined to be $\displaystyle\lim_{D\to \infty}H^*(CF^*(H),\delta)$ for either the non-degenerate case or Morse-Bott case, where the limit is taken w.r.t.\ the continuation maps.

\begin{remark}
    The construction above is a mixture of \cite[\S 2]{filling} for $H=0$ on $W$ and \cite{bourgeois2009symplectic} for autonomous Hamiltonians, and is a special case of the construction in \cite[\S 4.1]{divisor} where more general Morse-Bott families were considered.
\end{remark}

\subsection{Properties of symplectic cohomology}\label{ss24}
Symplectic cohomology $SH^*(W;R)$ has the following properties:
\begin{enumerate}
    \item\label{1} If we choose $H$ to be $C^2$ small and non-degenerate on $W$ and to be $h(r)$ on $\partial W \times (1,\infty)_r$ with $h''(r)>0$, then the periodic orbits of $X_H$ are either constant orbits on $W$ or non-constant orbits on $\partial W \times (1,\infty)$, which, in pairs after a small perturbation, correspond to Reeb orbits on $(\partial W,\lambda|_{\partial W})$.  Those constant orbits generate a subcomplex corresponding to the cohomology of $W$, and those non-constant orbits generate a quotient complex $CF^*_+(H)$, whose cohomology is called the positive symplectic cohomology $SH^*_+(W;R)$. Then we have a tautological long exact sequence,
    \begin{equation}\label{eqn:LES}
        \ldots \to H^*(W;R)\to SH^*(W;R)\to SH^{*}_+(W;R)\stackrel{\delta}{\to} H^{*+1}(W;R)\to \ldots
    \end{equation}
    In the Morse-Bott setup, $CF^*_+(H)$ is generated by $\check{\gamma},\hat{\gamma}$ and critical points of $g$ compute the cohomology of $W$.
    \item $SH^*(W;R)$ is a unital ring and $H^*(W;R)\to SH^*(W;R)$ is a unital ring map.
    \item If we consider $h(r)$ with $h'(r)=D$ for $r\in (1+\epsilon,+\infty)$ in the setup in \eqref{1}, assume there is no Reeb orbit with period $D$, the Hamiltonian-Floer cohomology defines filtered symplectic cohomology $SH^{*,<D}(W;R)$ and $SH^{*,<D}_+(W;R)$ with a similar tautological long exact sequence. And we have
    $$SH^*(W;R)=\varinjlim_{D\to \infty} SH^{*,<D}(W;R), \quad SH^*_+(W;R)=\varinjlim_{D\to \infty} SH^{*,<D}_+(W;R).$$
    \item We consider the map $\delta_{\partial}:SH^*_+(W;R)\to H^{*+1}(W;R)\to H^{*+1}(\partial W;R)$, on the cochain level, this can be defined by counting rigid configurations in the right of \Cref{fig:cascades} with the bottom flow line replaced by a gradient flow line in $\partial W$ (using an auxiliary Morse function $g_{\partial}$ on $\partial W$) viewed as fiber product over $W$, i.e.\ the right side of \Cref{fig:def}.
    \begin{figure}[H] {\small
    \begin{overpic}[scale=0.8]
    {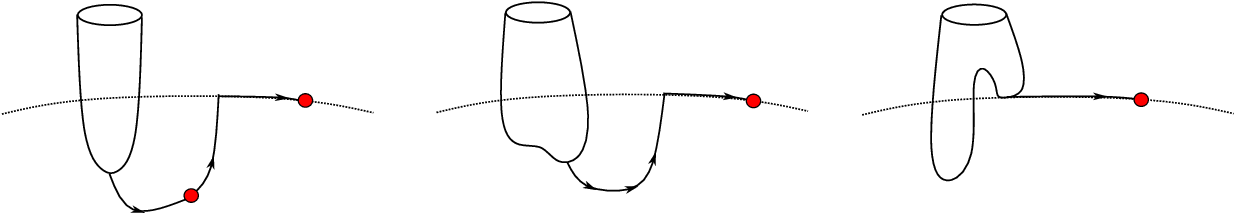}
    \put(2,8) {$\partial W$}
    \put (17,1) {$\in Crit(g)$}
    \put (25,7) {$\in Crit(g_{\partial})$}
    \put (20,10) {$\nabla g_{\partial}$}
    \put (18,6) {$\nabla g$}
    \put (50,5) {$\nabla g$}
    \put (11,-2) {$\nabla g$}
    \end{overpic}}
    \caption{Two descriptions of $\delta_{\partial}$ and the deformation}
    \label{fig:def}
    \end{figure} 
    \noindent
    The left side of \Cref{fig:def} is the composition of $SH^*_+(W;R)\to H^{*+1}(W;R)$ and $H^{*+1}(W;R)\to H^{*+1}(\partial W;R)$, which is homotopic to the right side through the middle deformation where the gradient flow lines of $\nabla g$ is defined on a finite interval,  see \cite[\S 3.1]{filling} for details. 
    \item For an exact subdomain $V\subset W$, i.e.\ $\lambda|_V$ makes $V$ into a Liouville domain, we have a Viterbo transfer map $SH^*(W;R)\to SH^*(V;R)$ compatible with all the structures, e.g.\ the tautological long exact sequence, the ring structure and so on. For the filtered version, we have 
    \begin{equation}\label{eqn:viterbo}
        SH^{*,<D}(W;R)\to SH^{*,<D}(V;R)
    \end{equation}
    see \cite[\S 5]{cieliebak2018symplectic}. 
\end{enumerate}

\subsection{Neck-stretching}
We first recall some basics of the neck-stretching procedure in \cite{bourgeois2003compactness}. We also recommend \cite[\S 2.3, 9.5]{cieliebak2018symplectic} for applications of neck-stretching in Floer theories.

We recall the setup of neck-stretching for the general case following \cite[\S 3.2]{zhou2019symplectic}.  Let $(W,\lambda)$ be an exact domain and $(Y,\alpha:=\lambda|_{Y})$ be a contact type hypersurface inside $W$.\footnote{The process works for strong filling $W$ as long as $Y$ is contact hypersurface.} The hypersurface divides $W$ into a cobordism $X$ union with a domain $W'$. Then we can find a small slice $([1-\eta,1+\eta]_r\times Y,\rd(r\alpha))$ symplectomorphic to a neighborhood of $Y$ in $W$. Assume $J|_{[1-\eta,1+\eta]_{r}\times Y}=J_0$, where $J_0$ is independent of $S^1$ and $r$ and $J_0(r\partial_{r})=R_\alpha,J_0\xi=\xi$ for $\xi:=\ker \alpha$. Then we pick a family of diffeomorphisms $\phi_R:[(1-\eta)e^{1-\frac{1}{R}}, (1+\eta)e^{\frac{1}{R}-1}]\to [1-\eta,1+\eta]$ for $R\in (0,1]$ such that $\phi_1=\Id$ and $\phi_R$ near the boundary is linear with slope $1$. Then the stretched almost complex structure $NS_{R}(J)$ is defined to be $J$ outside $[1-\eta,1+\eta]\times Y$ and is $(\phi_R\times \Id)_*J_0$ on $[1-\eta,1+\eta]\times Y$. Then $NS_{1}(J)=J$ and $NS_{0}(J)$ gives almost complex structures on the completions $\widehat{X}$, $\widehat{W'}$ and $\R_+\times Y$, which we will refer as the fully stretched almost complex structure.

We will consider the degeneration of curves solving the Floer equation with one positive cylindrical end asymptotic to a non-constant Hamiltonian orbit of $X_H$. Here we require that $H=0$ near the contact hypersurface $Y$. If the orbit is simple and $J$ is independent of $S^1$ or the orbit is multiply covered but $J$ depends on the $S^1$ coordinate near the orbit, the topmost curve in the SFT building, i.e.\ the curve in $\widehat{X}$, has the somewhere injectivity property. In particular, we can find regular $J$ on $\widehat{X}$ such that all relevant moduli spaces of the topmost curves, i.e.\ those in \S \ref{s4}. Note that we are not claiming the whole SFT building will be cut out transversely, just the top level which enjoys the somewhere injectivity.

\begin{figure}[H]
		\begin{tikzpicture}[scale=0.5]
		\path [fill=blue!15] (0,0) to [out=20, in=160]  (6,0) to [out=270,in=90] (6,-6) to [out=170,in=10] (0,-6) to [out=90, in=270] (0,0);
		\path [fill=red!15] (0,-6) to [out=10, in=170]  (6,-6) to [out=270,in=0] (3,-10) to [out=180,in=270] (0,-6); 
		\draw (0,0) to [out=20, in=160]  (6,0) to [out=270,in=90] (6,-6) to [out=170,in=10] (0,-6) to [out=90, in=270] (0,0);
		\draw (6,-6) to [out=270,in=0] (3,-10) to [out=180,in=270] (0,-6); 
		\draw[dashed] (0, -5.5) to [out=10, in=170] (6,-5.5);
		\draw[dashed] (0.02,-6.5) to [out=10, in=170] (5.98,-6.5);
		\draw[->] (1,-1) to (1.5,-1);
		\draw (1.5,-1) to (2,-1);
		\draw (2,-1) to [out=90, in=180] (2.5, -0.75) to [out=0, in = 90] (3,-1) to [out=270,in=0] (2.5,-1.25) to [out=180,in=270] (2,-1);
		\draw (2,-1) to [out=270,in=90] (1,-7) to [out=270,in=180] (2,-8) to [out=0,in=180](2.5,-3) to [out=0, in=90](3,-4);
		\draw (3,-1) to [out=270, in=90] (4,-4) to [out=270, in=0] (3.5,-4.25) to [out=180,in=270] (3,-4);
		\draw[dotted] (3,-4) to [out=90, in=180] (3.5,-3.75) to [out=0, in=90] (4,-4);
		\draw[->] (4,-4) to (4.25,-4);
		\draw (4.25,-4) to (4.5,-4);
		\draw (4.5,-4) to [out=90, in=180] (5,-3.75)  to [out=0, in=90] (5.5,-4) to [out=270,in=0] (5,-4.25) to [out=180,in=270] (4.5,-4);
		\draw (4.5,-4) to (4.5,-5);
		\draw (5.5,-4) to (5.5,-5);
		\draw[dotted]  (4.5,-5) to [out=90, in=180] (5,-4.75)  to [out=0, in=90] (5.5,-5);
		\draw (5.5,-5) to [out=270,in=0] (5,-5.25) to [out=180,in=270] (4.5,-5);
		\draw[->] (5.5,-5) to (5.75,-5);
		\draw (5.75,-5) to (5.8,-5);
		\node at (2.5,-2) {$u_1$};
		\node at (5,-4.6) {$u_2$};
		\end{tikzpicture}
		\hspace{1cm}
		\begin{tikzpicture}[xscale=0.5,yscale=0.7]
		\path [fill=blue!15] (0,-1) to [out=20, in=160]  (6,-1) to [out=270,in=90] (6,-6) to [out=170,in=10] (0,-6) to [out=90, in=270] (0,-1);
		\path [fill=red!15] (0,-6) to [out=10, in=170]  (6,-6) to [out=270,in=0] (3,-10) to [out=180,in=270] (0,-6); 
		\draw (0,-1) to [out=20, in=160]  (6,-1) to [out=270,in=90] (6,-6) to [out=170,in=10] (0,-6) to [out=90, in=270] (0,-1);
		\draw (6,-6) to [out=270,in=0] (3,-10) to [out=180,in=270] (0,-6);
		\draw[dashed] (0, -5.5) to [out=10, in=170] (6,-5.5);
		\draw[dashed] (0.1, -6.5) to [out=10, in=170] (5.9,-6.5);
		\draw[->] (1,-2) to (1.5,-2);
		\draw (1.5,-2) to (2,-2);
		\draw (2,-2) to [out=90, in=180] (2.5, -1.8) to [out=0, in = 90] (3,-2) to [out=270,in=0] (2.5,-2.2) to [out=180,in=270] (2,-2);
		\draw (2,-2) to [out=270,in=90] (1,-7) to [out=270,in=180] (2,-8) to [out=0,in=180](2.5,-3.5) to [out=0, in=90](3,-4);
		\draw (3,-2) to [out=270, in=90] (4,-4) to [out=270, in=0] (3.5,-4.2) to [out=180,in=270] (3,-4);
		\draw[dotted] (3,-4) to [out=90, in=180] (3.5,-3.8) to [out=0, in=90] (4,-4);
		\draw[->] (4,-4) to (4.25,-4);
		\draw (4.25,-4) to (4.5,-4);
		\draw (4.5,-4) to [out=90, in=180] (5,-3.75)  to [out=0, in=90] (5.5,-4)  to [out=270,in=0] (5,-4.25) to [out=180,in=270] (4.5,-4);
        \draw (4.5,-4) to (4.5,-5);
		\draw (5.5,-4) to (5.5,-5);
		\draw[dotted]  (4.5,-5) to [out=90, in=180] (5,-4.75)  to [out=0, in=90] (5.5,-5);
		\draw (5.5,-5) to [out=270,in=0] (5,-5.25) to [out=180,in=270] (4.5,-5);
		\draw[->] (5.5,-5) to (5.75,-5);
		\draw (5.75,-5) to (5.8,-5);
		\node at (2.5,-3) {$u_1$};
		\node at (5,-4.5) {$u_2$};
		\end{tikzpicture}
		\hspace{1cm}
		\begin{tikzpicture}[scale=0.5]
		\path [fill=blue!15] (0.5,-6) to [out=90, in=270]  (0,0) to [out=20,in=160] (6,0) to [out=270,in=90] (5.5,-6) to [out=160, in=20] (0.5,-6);
		\path [fill=purple!15] (0.5,-6.2) to [out=20,in=160] (5.5,-6.2) to [out=270,in=90] (5.5,-12) to [out=160, in=20] (0.5,-12) to [out=90, in=270] (0.5,-6.2);
		\path [fill=red!15] (0.5, -12.2) to [out=20,in=160] (5.5,-12.2) to [out=270,in=0] (3,-16) to [out=180,in=270] (0.5,-12.2);
		\draw (0.5,-6) to [out=90, in=270]  (0,0) to [out=20,in=160] (6,0) to [out=270,in=90] (5.5,-6);
		\draw [dashed] (5.5,-6) to [out=160, in=20] (0.5,-6);
		\draw[dashed] (0.5,-6.2) to [out=20,in=160] (5.5,-6.2);
		\draw (5.5,-6.2) to [out=270,in=90] (5.5,-12);
		\draw[dashed] (5.5,-12) to [out=160, in=20] (0.5,-12);
		\draw (0.5,-12) to [out=90, in=270] (0.5,-6.2);
		\draw [dashed](0.5, -12.2) to [out=20,in=160] (5.5,-12.2); 
		\draw (5.5,-12.2)to [out=270,in=0] (3,-16) to [out=180,in=270] (0.5,-12.2);
		\draw[->] (1,-1) to (1.5,-1);
		\draw (1.5,-1) to (2,-1);
		\draw (2,-1) to [out=90, in=180] (2.5, -0.75) to [out=0, in = 90] (3,-1) to [out=270,in=0] (2.5,-1.25) to [out=180,in=270] (2,-1);
		\draw (2,-1) to [out=270,in=90] (1,-4) to [out=270, in=180]  (1.5,-5.8) to [out=0,in=270](2,-5) to [out=90,in=180] (2.5,-3) to [out=0, in=90](3,-4);
		\draw (3,-1) to [out=270, in=90] (4,-4) to [out=270, in=0] (3.5,-4.25) to [out=180,in=270] (3,-4);
		\draw[dotted] (3,-4) to [out=90, in=180] (3.5,-3.75) to [out=0, in=90] (4,-4);
		\draw[->] (4,-4) to (4.25,-4);
		\draw (4.25,-4) to (4.5,-4);
		\draw (4.5,-4) to [out=90, in=180] (4.9,-3.75)  to [out=0, in=90] (5.3,-4) to [out=270,in=0] (4.9,-4.25) to [out=1800,in=270] (4.5,-4);
		\draw (4.5,-4) to (4.5,-5);
		\draw (5.3,-4) to (5.3,-5);
		\draw[dotted] (4.5,-5) to [out=90, in=180] (4.9,-4.75)  to [out=0, in=90] (5.3,-5);
		\draw (5.3,-5) to [out=270,in=0] (4.9,-5.25) to [out=1800,in=270] (4.5,-5);
        \draw[->] (5.3,-5) to (5.5,-5);
		\draw (1.5,-5.9) to [out=180,in=90] (0.8,-9) to [out=270, in=180] (1.5,-11.8) to [out=0,in=270] (2.2,-9) to [out=90, in=0] (1.5,-5.9);
		\draw (1.5,-11.9) to [out=180,in=90] (0.8,-13) to [out=270, in=180] (1.5,-14) to [out=0,in=270] (2.2,-13) to [out=90, in=0] (1.5,-11.9);

		\node at (1.5,-5.8) [circle, fill=white, draw, outer sep=0pt, inner sep=3 pt] {};

		\node at (1.5,-11.8) [circle, fill=white, draw, outer sep=0pt, inner sep=3 pt] {};
		\node at (2.5,-2) {$u^{\infty}_1$};
		\node at (5,-4.7) {$u^{\infty}_2$};
		\node at (4.5,-2) {$\widehat{X}$};
		\node at (4,-10) {$Y\times \R_+$};
		\node at (4,-14) {$\widehat{W'}$};
		\end{tikzpicture}
    \caption{Neck-stretching for cascades model, the top level is contained in the blue region.}
    \label{fig}
\end{figure}
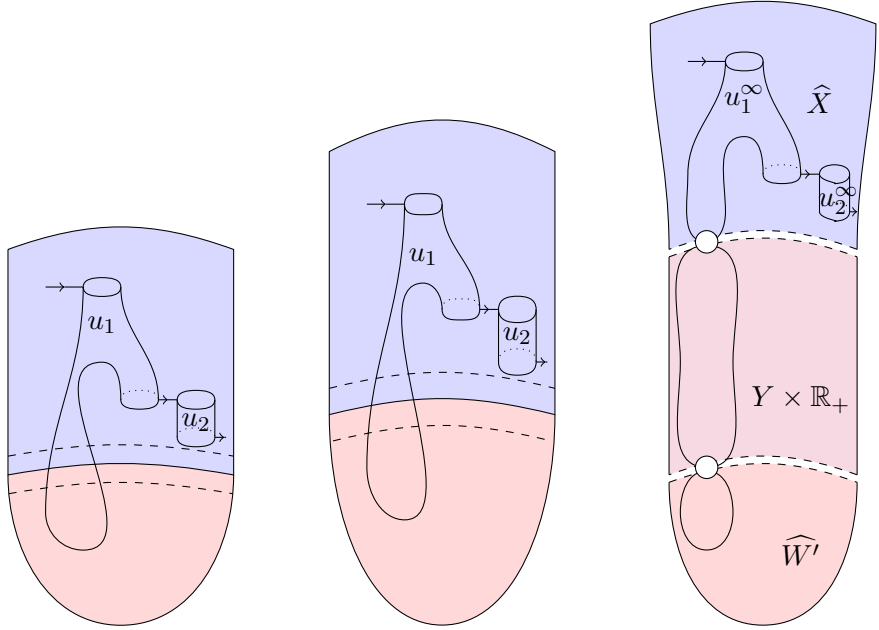
In the figure, we use $\bigcirc$ to indicate the puncture that is asymptotic to a Reeb orbit. For the compactification of curves in the topmost SFT level, in addition to the usual SFT building in the symplectization $\R_+\times Y$ stacked from below \cite{bourgeois2003compactness}, we also need to include Hamiltonian-Floer breakings near the cylindrical ends. If we use autonomous Hamiltonians and cascades,  we need to include curves with multiple cascades levels and their degeneration, e.g.\ $l_i=0,\infty$ in the cascades for some horizontal level $i$. A generic configuration is described in the top-right of the figure above, but we could also have more cascades levels with the connecting Morse trajectories degenerate to $0$ length or broken Morse trajectories.

A useful fact from the non-negativity of energy is the following action constraint. Let $u$ be a Floer cylinder in $\widehat{X}$ with negative punctures asymptotic to a multiset $\Gamma$ of Reeb orbits (i.e.\ a set of Reeb orbits with possible duplications). Assume $\displaystyle\lim_{s\to \infty} u = x$ and $\displaystyle\lim_{s\to -\infty} u = y$,  then we have
\begin{equation}\label{eqn:positive}
\cA_H(y)-\cA_H(x)-\sum_{\gamma\in \Gamma}\int \gamma^*\alpha \ge 0
\end{equation}
where $\cA_H$ is \eqref{eqn:action}, c.f.\ \cite[Lemma 2.4]{cieliebak2018symplectic}.

If we apply neck-stretching to the contact boundary $Y=\partial W$ in the completion $\widehat{W}$, assume $H=h(r)$ on $(0,+\infty) \times \partial W$ and we use the cascades model and a cylindrical convex almost complex structure, for a top-level curve as above, we get 
\begin{equation}\label{eqn:positive2}
\int \gamma_x^*\alpha -\int \gamma_y^*\alpha-\sum_{\gamma\in \Gamma}\int \gamma^*\alpha \ge 0
\end{equation}
where $\gamma_x,\gamma_y$ are the corresponding Reeb orbits for $x,y$. This follows from the non-negativity of $\int u^*\pi^*\rd \alpha$, where $\pi:\widehat{X}= (0,+\infty)\times Y\to Y$. 

%% file: s3.tex
\section{Topologically simple fillings}\label{s3}
Exploiting the independence of augmentations using grading constraints was initiated by Bourgeois and Oancea \cite{MR2471597}, also see the work of Cieliebak and Oancea \cite{cieliebak2018symplectic}, and Uebele \cite{MR4031538}, where they introduced the notion of index-positive contact manifolds (\Cref{rmk:ADC}). This notion was generalized by Lazarev \cite{lazarev2016contact} to the notion of asymptotically dynamically convex (ADC) manifolds to contain examples like flexibly fillable contact manifolds with vanishing first Chern class.

\begin{definition}[{\cite[Definition 3.6.]{lazarev2016contact}}]\label{def:ADC}
    Let $(Y^{2n-1},\xi)$ be a co-oriented contact manifold such that $c_1^{\Q}(\xi)=0$, we say $Y$ is ADC if there exist contact forms $\alpha_1\ge \alpha_2\ge \ldots$ and $0<D_1<D_2<\ldots \to \infty$, such that all contractible Reeb orbits of $\alpha_i$ with period at most $D_i$ are non-degenerate and $\mu_{\CZ}+n-3>0$.
\end{definition}
Several structural maps on ($S^1$-equivariant) symplectic cohomology of exact fillings of ADC manifolds are independent of topologically simple (a stronger topologically simple condition than the one used in this paper, namely $\pi_1(Y)\to \pi_1(W)$ is injective and $c_1(W)=0$) fillings \cite{filling,zhou2019symplectic}. Those topological conditions are used to get a $\Z$ grading for the symplectic cohomology generated by \emph{contractible orbits}, as the ADC condition only requires that $\mu_{\CZ}(\gamma)+n-3>0$ for a contractible Reeb orbit $\gamma$. To weaken the topological conditions, we need to use stronger ADC conditions. 
\begin{definition}\label{def:ADC'}
    Let $(Y^{2n-1},\xi)$ be a co-oriented contact manifold such that $c_1^{\Q}(\xi)=0$, we say $Y$ is ADC' if there exist contact forms $\alpha_1\ge \alpha_2\ge \ldots$ and $0<D_1<D_2<\ldots \to \infty$, such that all Reeb orbits of $\alpha_i$ with \emph{torsion homology class} and period at most $D_i$ are non-degenerate and $\mu_{\CZ}+n-3>0$.
\end{definition}
These different notions of ADC in \Cref{def:ADC,def:ADC'} result in slightly different theorems. For example, if we use ADC in \Cref{def:ADC}, then $SH_+^*(W)$ generated by \emph{contractible} Reeb orbits is independent of exact filling $W$ as long as $c_1(W)=0$ and $\pi_1(Y)\to \pi_1(W)$ is injective. However, if we use \Cref{def:ADC'}, then $SH_+^*(W)$ generated by Reeb orbits of \emph{torsion homology class} is independent of exact filling $W$ as long as $c_1^{\Q}(W)=0$ and $H_1(Y;\Q)\to H_1(W;\Q)$ is injective, i.e.\ topologically simple in this paper. This change can be applied to all results in \cite{filling,zhou2019symplectic}. Since flexibly fillable contact manifolds with torsion first Chern class is ADC' in the sense of \Cref{def:ADC'} by \Cref{prop:reeb}, therefore we use $c_1^{\Q}(W)=0$ and $H_1(Y;\Q)\to H_1(W;\Q)$ injective as the topologically simple condition in this paper and all results in \cite{filling,zhou2019symplectic} stated for flexibly fillable contact manifolds hold with such weaker topologically simple condition.

\begin{example}\label{ex:quotient}
    Let $G\subset U(n)$ acting freely on $S^{2n-1}$ and $n\ge 2$, since $G$ preserves the Euclidean metric and complex structure on $\C^n$, $G$ preserves the standard Liouville form $-\frac{1}{2}r\rd r\circ J$. This gives a contact structure $\xi_{\std}$ on $S^{2n-1}/G$, whose universal cover is the standard contact sphere $(S^{2n-1},\xi_{\std})$. Therefore $(S^{2n-1}/G,\xi_{\std})$ is always ADC in the sense of \Cref{def:ADC},  as contractible orbits of $(S^{2n-1}/G,\xi_{\std})$ are the same as those on $(S^{2n-1},\xi_{\std})$. However,  by \cite[Theorem A]{quotient}, there are no topological simple fillings in the stronger sense in \cite{filling} ($\pi_1$ injective and vanishing of the first Chern class)  of  $(S^{2n-1}/G,\xi_{\std})$.
    
    On the other hand, when we consider all Reeb orbits on $(S^{2n-1}/G,\xi_{\std})$, all of which are torsion. The highest minimal SFT degree is shown to be twice of the minimal discrepancy number of the singularity $\C^n/G$ whenever the latter is non-negative by \cite[Theorem 1.1.]{Reeb}. Therefore $(S^{2n-1}/G,\xi_{\std})$ is ADC' in the sense of \Cref{def:ADC'} if and only if the minimal discrepancy number of the singularity $\C^n/G$ is positive, i.e.\ $\C^n/G$ is terminal.

    For example, $(\mathbb{RP}^3,\xi_{\std})$ is ADC in the sense of \Cref{def:ADC} but is not ADC' in the sense of \Cref{def:ADC'}.  $(\mathbb{RP}^3,\xi_{\std})$ has two exact fillings, namely $T^*S^2$ and the \emph{orbifold} filling $\C^2/(\Z/2)$. Symplectic cohomology can be defined for exact orbifolds with similar properties, c.f.\ \cite{gironella2021exact}. $T^*S^2$ is topologically simple in this paper's sense, while $\C^2/(\Z/2)$ is topologically simple in the stronger sense in \cite{filling}, namely $\pi_1(\mathbb{RP}^3)\to \pi_1^{\mathrm{orb}}(\C^2/(\Z/2))$ is injective, where $\pi_1^{\mathrm{orb}}(\C^2/(\Z/2))$ is the orbifold fundamental group \cite[\S 2.5]{gironella2021exact}. Indeed, $T^*S^2$ and $\C^2/(\Z/2)$ are quite different from the symplectic cohomology perspective, namely $SH^*(T^*S^2;\Q)$ has infinite rank while $SH^*(\C^2/(\Z/2);\Q)=0$ \cite[Theorem B]{gironella2021exact}.

    When $n\ge 3$,  $(\mathbb{RP}^{2n-1},\xi_{\std})$ is ADC' in the sense of \Cref{def:ADC'}, motivated from \cite{quotient},  $(\mathbb{RP}^{2n-1},\xi_{\std})$ is expected to have an unique exact \emph{orbifold} filling when $n\ge 3$, namely $\C^n/(\Z/2)$.
\end{example}

\begin{remark}[Clarification on asymptotically dynamical convexity and dynamical convexity]\label{rmk:ADC}
Dynamical convexity was introduced by Hofer, Wysocki, and Zehnder \cite{MR1652928} as a dynamical consequence of convexity, namely a contact form $\alpha$ on $(S^{2n-1},\xi_{std})$ is dynamical convex if all Reeb orbits have Conley-Zehnder indices at least $n+1$. It was shown recently by Chaidez and Edtmair \cite{MR4438355,chaidez2022ruelle} that dynamical convexity does not imply convexity. In terms of cylindrical contact homology, dynamical convexity means that there are no Reeb orbits of small Conley-Zehnder indices that are not visible in the cylindrical contact homology. This is the perspective used by Abreu and Macarini in \cite{MR3694649} for contact manifolds other than spheres. On the other hand, Bourgeois, Cieliebak,  Oancea, and Uebele \cite{MR2471597,cieliebak2018symplectic,MR4031538} introduced index positivity of a contact form, namely if the SFT degree $\mu_{\CZ}+n-3$ of any contractible Reeb orbit is positive. For $(S^{2n-1},\xi_{std})$, this means that the Conley-Zehnder indices are at least $2-n \ne n+1$. From this perspective, the ADC condition is ``asymptotical" index positivity by considering a sequence of contact forms instead of just one.  
\end{remark}

To get results with even weaker topological assumptions, we need properties on Conley-Zehnder indices for general Reeb orbits. In particular, we have the following:
\begin{definition}\label{def:ADC-pi_1}
Let $(Y,\xi)$ be a contact manifold such that $c_1^{\Q}(\xi)=0$. Let $\Psi$ be a trivialization of $\det_{\C}\oplus^N\xi$ for some $N\in \N_+$. We say $(Y,\xi,\Psi)$ is \emph{$\pi_1$-sensitive ADC} if there exist contact forms $\alpha_1>\alpha_2>\ldots$, positive real numbers $D_1<D_2<\ldots$ converging to infinity, such that \emph{all} Reeb orbits of $\alpha_i$ of period up to $D_i$ are non-degenerate and have rational SFT grading $\mu_{\CZ}(\gamma)+n-3>0$. We say $(Y,\xi,\Psi)$ is \emph{$\pi_1$-sensitive TADC}, if in addition, there is a contact form $\alpha$ such that all $\alpha_i>\alpha$, where `T' stands for tamed \cite[Definition 6.2]{filling}. 
\end{definition}
Since the Conley-Zehnder index of Reeb orbits of torsion homology class is independent of the trivialization $\Psi$, we have, for any trivialization $\Psi$.
$$\pi_1\text{-sensitive ADC}\Rightarrow \text{ADC'} \Rightarrow \text{ADC}.$$
In general, if $H^1(Y;\Q)\ne 0$, then the notion of generalized ADC depends on the trivialization $\Psi$.

\begin{example}\label{ex:generalized_ADC}
We have the following examples of $\pi_1$-sensitive ADC contact manifolds.
\begin{enumerate}
   \item Let $G\subset U(n)$ such that the quotient $\C^m/G$ has an isolated singularity at $0$, the contact link $(S^{2n-1}/G,\xi_{\std})$ is $\pi_1$-sensitive ADC if and only if $\C^n/G$ is a terminal singularity by the work of McLean \cite{Reeb}, as explained in \Cref{ex:quotient}
    \item The contact boundary of a flexible Weinstein domain with vanishing rational first Chern class for any trivialization by the arguments in \cite{lazarev2016contact}. More precisely, the contact boundary of a subcritical Weinstein domain is generalized ADC as all of the relevant orbits can be assumed to be contractible. When we attach a flexible handle, non-contractible orbits could appear, however, the argument of lifting the Conley-Zehnder indices by adding zig-zags in \cite[Theorem 3.18]{lazarev2016contact} works for non-contractible orbits and any fixed trivialization. 
    \item For a closed manifold $Q$, we have that $\det_{\C} \oplus^2 TT^*Q$ is trivialized using the trivial real bundle $\det_{\R}\oplus^2TQ$. We use $\Psi$ to denote the trivialization. Then $(S^*Q,\Psi)$ is generalized ADC if $\dim Q\ge 4$, as the Conley-Zehnder index using such trivialization is the Morse index when the contact form is induced from a metric. This is an example where the notion of generalized ADC depends on $\Psi$, as changing $\Psi$ will increase the Conley-Zehnder indices of some orbits with nontrivial homotopy classes and decrease the same amount for orbits with the opposite homotopy classes, e.g.\ $T^*T^n$. The same holds for any closed orbifold $Q$ with only isolated singularities (then $S^*Q$ is a contact manifold).
    \item Let $V$ be a Liouville domain, such that $c_1^{\Q}(V)=0$, then $\partial (V\times \D)$ is generalized ADC for any trivialization by the proof of \cite[Theorem K]{filling}. 
\end{enumerate}
\end{example}

\begin{proof}[Proof of \Cref{prop:simple}]
If $c^\Q_1(W)=0$, then we can trivialize $\det \oplus^N TW$ for some $N\in \N_+$. Note that $H^1(W^{flex};\Z)\to H^1(Y;\Z)$ is an isomorphism, hence there is a trivialization of $\det_{\C} \oplus^N TW^{flex}$ whose restriction to $Y=\partial W^{flex}$ is the same as the restriction of the trivialization of $\det \oplus^N TW$. Since $Y$ is $\pi_1$-sensitive ADC for any trivialization, we then run the argument of \cite[Corollary B]{filling} using such trivializations and conclude the claim.
\end{proof}

\begin{proposition}\label{prop:spectral_sequence}
Assume $(Y,\xi,\Psi)$ is $\pi_1$-sensitive ADC and there is an exact filling $W$, such that $\Psi$ extends to a trivialization $\widetilde{\Psi}$ of $\det_{\C} \oplus^NTW$. Then for any exact filling $V$ of $Y$, there is a spectral sequence converging to $SH^*_+(V;R)$ (not graded), such that 
\begin{enumerate}
    \item The $(N+1)$th page of the spectral sequence is isomorphic to $SH^*_+(W;R;\widetilde{\Psi})$ (filtered by the $\Q$-grading using $\widetilde{\Psi}$) for any coefficient ring $R$.
    \item\label{SS_mor} There is a spectral sequence on the Morse cochain complex on $Y$, such that the cochain map $\delta_{\partial}$ from the positive cochain complex to the Morse cochain complex of $Y$ is compatible with the filtration/spectral sequence. On the $(N+1)$th page, the induced map is isomorphic to $SH^*_+(W;R;\widetilde{\Psi})\to H^{*+1}(Y;R)$ and the spectral sequence on the Morse cochain complex on $Y$ degenerates on the $(N+1)$th page.
\end{enumerate}
If  $(Y,\xi,\Psi)$ is $\pi_1$-sensitive TADC, the same holds for (semi-positive) strong fillings $V,W$ and $R$ the Novikov field. 
\end{proposition}
\begin{proof}
First note that $\mu_{\CZ}(x)$ computed using the trivialization $\Psi$ is always a multiple of $\frac{1}{N}$. Hence the $\pi_1$-sensitive ADC property implies that $\mu_{\CZ}+n-3\ge \frac{1}{N}$. The proof follows from applying arguments in \cite[\S 3]{quotient} to the spectral sequence associated to the filtration
$$F^kCF_+(H):=\left\langle x\left| |x|^{\partial}\ge \frac{k}{N} \right.\right\rangle, \quad k\in \Z, $$
where $|x|^{\partial}=n-\mu_{\CZ}(x)$ defined using the trivialization $\Psi$. By the same argument of \cite[Proposition 3.3]{quotient}, the differential is compatible with filtration by neck-stretching, and moreover, there is no differential before the $N$th page, and on $N$th page, there are differentials from $x$ to $y$ with $|y|^{\partial}-|x|^{\partial}=1$, whose underlying curve is contained in the cylindrical end of the completion for a sufficiently stretched almost complex structure. This differential computes $SH^*_+(W;R;\widetilde{\Psi})$, yielding the first claim. The Morse cochain complex of $Y$ has a similar filtration by 
$$F^kC^*(g_{\partial}):=\left\langle p\in Crit(g_{\partial})\left| \ind_{\mathrm{Mo}}(p)\ge \frac{k}{N}+1 \right.\right\rangle, \quad k\in \Z, $$
where $g_{\partial}$ is an auxiliary Morse function on $Y$ in the definition of $\delta_{\partial}$ in \S \ref{ss24} and $\ind_{\mathrm{Mo}}$ is the Morse index. The second claim follows from the proof of \cite[Proposition 3.4]{quotient}. 

Strictly speaking, we need to apply the above argument to the infinite telescope construction, since the above spectral sequence only works for $CF_+^{*,<D_i}(\alpha_i)$ (the positive symplectic cochain complex of a strict filling of $(Y,\alpha_i)$ using a Hamiltonian with slope $D_i$). The full positive symplectic cohomology is computed as the $\varinjlim CF_+^{*,<D_i}(\alpha_i)$, where the connecting morphism is the Viterbo transfer map \cite[Proof of Proposition 3.8]{lazarev2016contact}. The direct limit has a model 
\begin{equation}\label{eqn:limit}
    Cone\left(\oplus_i CF_+^{*,<D_i}(\alpha_i)\stackrel{1-\nu}{\longrightarrow} \oplus_iCF_+^{*,<D_i}(\alpha_i)\right)
\end{equation}
where $Cone$ stands for mapping cone and $\nu:\oplus_i CF_+^{*,<D_i}(\alpha_i)\to \oplus_i CF_+^{*,<D_i}(\alpha_i)$ is induced from $ CF_+^{*,<D_i}(\alpha_i)\to  CF_+^{*,<D_{i+1}}(\alpha_{i+1})$, see \cite[Definition 28, Remark 24]{murfet2006derived}. We can apply the above spectral sequence argument to \eqref{eqn:limit}.
\end{proof}

\begin{corollary}\label{cor:bound}
Let $(Y^{2n-1},\xi)$ be the contact boundary of a flexible Weinstein domain $W^{flex}$ with $c^{\Q}_1(W^{flex})=0$. Then for any exact filling $W$ of $Y$, we have $\dim \oplus SH^*_+(W;\Q)\le \dim \oplus H^*(W^{flex};\Q)$.
\end{corollary}
\begin{proof}
By \Cref{ex:generalized_ADC}, there is a trivialization $\Psi$, such that $(Y,\xi,\Psi)$ is $\pi_1$-sensitive ADC. Since the trivialization $\Psi$ is the restriction of a trivialization of $\det_{\C}\oplus^N TW^{flex}$, by \cite{BEE,Subflexible}, we have $SH_+^*(W^{flex};\Q)=H^{*+1}(W^{flex},\Q)$. Then by \Cref{prop:spectral_sequence}, we have $\dim \oplus SH^*_+(W;\Q)\le \dim \oplus H^*(W^{flex};\Q)$ for any exact filling $W$. 
\end{proof}

In general, a morphism between two spectral sequences in \eqref{SS_mor} of \Cref{prop:spectral_sequence} only captures the morphism on the associated grade of the limits, i.e.\ the leading terms. But we can exploit the tautological degeneracy of the spectral sequence on the Morse cochain complex of $Y$ in some special cases, which improves some of the results in \cite{filling}.

\begin{proof}[Proof of \Cref{prop:SS}]
For \eqref{c1}, by assumption, we have $SH^{*-1}_+(W^{flex};\Q)\simeq H^*(W^{flex};\Q)\to H^*(Y;\Q)$ is injective. Then by \Cref{prop:spectral_sequence}, the spectral sequence map from the spectral sequence of $SH^*_+(W;\Q)$ to that of $H^*(Y;\Q)$, on the $(N+1)$th page, is isomorphic to the injective map $SH^*_+(W^{flex};\Q)\to H^{*+1}(Y;\Q)$. Since the spectral sequence on $H^*(Y;\Q)$ of index gap $1/N$ degenerates from the $(N+1)$ page, the injectivity implies that the spectral sequence on $SH^*_+(W;\Q)$  also degenerates at the $(N+1)$ page. And the map $SH^*(W;\Q)\to H^{*+1}(Y;\Q)$ on the associated graded is the same as $SH^*_+(W_0;\Q)\to H^{*+1}(Y;\Q)$. Since $1$ is in the image of  $SH^*_+(W^{flex};\Q)\to H^{*+1}(Y;\Q)$, we know that $1+a$ is in the image of  $SH^*_+(W;\Q)\to H^{*+1}(Y;\Q)$ for $\deg(a)>0$. Therefore, we have $SH^*(W;\Q)=0$ and  $SH^*_+(W;\Q)\simeq H^{*+1}(W;\Q)$. The injectivity of $SH^*_+(W^{flex};\Q)\to H^{*+1}(Y;\Q)$ implies that  $SH^*_+(W;\Q)\simeq H^{*+1}(W;\Q) \to  H^{*+1}(Y;\Q)$ is also injective. As a consequence, we have $c_1^{\Q}(W)=0$. Then the claim follows from \Cref{prop:simple}.

For \eqref{c2}, when $n$ is even,  since $W^{flex}$ is Weinstein, i.e.\ it has a handle decomposition of $k$-handles for $k\le n$, we have $H^*(W^{flex};\Z)\to H^{*}(Y;\Z)$ is injective for $*\le n-1$. In particular, we have $H^*(W^{flex};\Z)\to H^{*}(Y;\Z)$ is injective on odd degrees. Note that symplectic cohomology is canonically graded by $\Z/2$, and the differentials on the spectral sequence are compatible with the $\Z/2$ grading. By looking at the $(N+1)$th page of the spectral sequence map as before, which is injective on even ($\Z/2$) degrees of $SH^*_+(W;\Z)$, the differential of the $(N+1)$th page of the spectral sequence for $SH^*_+(W;\Z)$ must be zero on odd degrees. As a consequence, by induction, the $(N+k)$th page of the spectral sequence map is injective on even degrees, and the differential of the $(N+k)$th page of the spectral sequence for $SH^*_+(W;\Z)$ must be zero on odd degrees for $k>0$. As a consequence, we also have that $1$ is in the image of the morphism of spectral sequences in \eqref{SS_mor} of \Cref{prop:spectral_sequence} on the $\infty$th page, hence $SH^*(W;\Z)=0$ and $SH^*_+(W;\Q)\simeq H^{*+1}(W;\Q)$. Then \Cref{cor:bound} implies that $\dim \oplus_{*=1}^{2n}H^*(W;\Q)\le \dim \oplus_{*=1}^{2n}H^*(W^{flex};\Q)$.

For \eqref{a}, we claim that the spectral sequence of $SH_+^*(W;\Q)$ degenerates at the $(N+1)$th page. For otherwise, the non-trivial differential must be from even degrees of $SH^*_+(W;\Q)$ to odd degrees by the argument in \eqref{c2}. However, this implies that the total dimension of even degrees of $SH^*_+(W;\Q)$ is smaller than that of $SH^*_+(W^{flex};\Q)$, which is the total dimension of $\oplus_{i=1}^{n} H^{2i+1}(W^{flex};\Q) \simeq \oplus_{i=1}^{n} \Ima (H^{2i+1}(W^{flex};\Q)\to H^{2i+1}(Y;\Q))$.  By \cite[the proof of Proposition 4.3. Case I]{bowden2022making}\footnote{Or Theorem 37 of the first version of \cite{bowden2022making} available at \href{https://arxiv.org/pdf/2211.03680v1.pdf}{https://arxiv.org/pdf/2211.03680v1.pdf}.}, we have $H^*(W;\Q)\to H^*(Y;\Q)$ is surjective onto the image of $H^*(W^{flex};\Q)\to H^*(Y;\Q)$ for $2\le * \le n$. As $H^1(W^{flex};\Q)=0$, we know that $\dim \oplus_{i=1}^{n} H^{2i+1}(W;\Q)\ge \dim \oplus_{i=1}^{n} H^{2i+1}(W^{flex};\Q)$. Then this contradicts with that $SH^*_+(W;\Q)\to H^{*+1}(W;\Q)$ is surjective (for $*$ even). Therefore the spectral sequence degenerates at the $(N+1)$th page and we have $SH^*_+(W;\Q)\simeq  SH^*_+(W^{flex};\Q)$, and the claim follows from that $SH^*(W;\Q)=0$ in \eqref{c2}.

For \eqref{b}, we already know that the spectral sequence of $SH_+^*(W;\Q)$ degenerates at the $(N+1)$th page. We claim that $H^2(W;\Q)\to H^2(Y;\Q)$ is injective, for otherwise, we have $\dim H^2(W,Y;\Q)\ge 1$. Then by Lefschetz duality and the universal coefficient theorem, we have $\dim H^{2n-2}(W;\Q)\ge 1$. As a consequence, we have $\dim \oplus_{*=0}^{2n} H^{*}(W;\Q)\ge 2+\dim \oplus_{*=0}^{n}\Ima(H^{*}(W^{flex};\Q)\to H^*(Y;\Q))$. When $\dim \ker (H^n(W^{flex};\Q)\to H^n(Y;\Q))=1$, we have $\dim \oplus_{*=0}^{2n}H^*(W^{flex}) = 1+ \oplus_{*=0}^{n}\Ima(H^{*}(W^{flex};\Q)\to H^*(Y;\Q))$. Hence we arrive at a contradiction with \eqref{c2}. Now since  $H^2(W;\Q)\to H^2(Y;\Q)$ is injective, we have $c_1^{\Q}(W)=0$ and we can apply \Cref{prop:simple}.
\end{proof}

The main difference between \Cref{prop:SS} and results in \cite{filling} is that we can prove topological simplicity instead of assuming it in \Cref{prop:SS}, but we need addition information, e.g.\ $Y$ being flexibly fillable,  to get the degeneracy of the spectral sequence. When $n$ is odd, a priori, there could be differentials in the spectral sequence acting non-trivially on the element in $SH^*_+(W)$ that is supposed to kill the unit. The $H^n(W^{flex};\Q)\to H^n(Y;\Q)$ injective condition prevents such differentials by exploiting the tautological degeneracy on the spectral sequence on $H^*(Y)$. It is possible to strengthen \eqref{b} to the case where the intersection form on the cokernel of $H_n(Y;\Q)\to H_n(W^{flex};\Q)$ is positive/negative definite. 

%% file: s4.tex
\section{Proof of \Cref{thm:main}}\label{s4}
Let $Y$ be a flexibly fillable contact manifold with $c^{\Q}_1(\xi)=0$ and $W$ a topologically simple exact filling. We pick a contact form $\alpha_D$ as in \Cref{prop:reeb}. Let $\gamma_1,\ldots,\gamma_N$ denote all the Reeb orbits of action smaller than $D$ with Conley-Zehnder index $1$. We assume they are ordered increasingly with respect to their periods. Let $(W_D,\lambda_D)$ be the strict topologically simple exact filling of the strict contact manifold $(Y,\alpha_D)$ that is homotopic to $W$. In the definition of filtered positive symplectic cohomology $SH^{*,<D}_+(W_D)$ with slope $D$, we will use the following special Hamiltonian $H$. 
\begin{enumerate}
	\item $H=0$ on $W_D$ and $H'(r)=D$ for $r>1+w$ for $w>0$.  
	\item $H$ on $[1,1+w]\times Y$ is a small perturbation of $H=f(r)$ with $f''(r)>0$ such that the periodic orbits of $X_H$ are non-degenerate and in a two-to-one correspondence with Reeb orbits of period smaller than $D$.
\end{enumerate}
More precisely, every non-degenerate Reeb orbits $\gamma$ will split into two Hamiltonian orbits $\hat{\gamma}$ and $\check{\gamma}$ with $\mu_{\CZ}(\hat{\gamma}) = \mu_{\CZ}(\gamma)+1$ and $\mu_{\CZ}(\check{\gamma})=\mu_{\CZ}(\gamma)$ following \cite{bourgeois2009symplectic}. Here we slightly abuse the notion as $\check{\gamma},\hat{\gamma}$ already appeared in the Morse-Bott setup in \S \ref{ss23}. This is for the simplicity of notation and is conceptually consistent, as the non-degenerate Hamiltonian orbits $\check{\gamma},\hat{\gamma}$ here corresponds to the critical points on $\overline{\gamma}$ used in \S \ref{ss23} through perturbation by \cite[Lemma 3.4]{bourgeois2009symplectic}. When we view them as generators of $CF_+^*(H)$, since $[\check{\gamma}_1], \ldots, [\check{\gamma}]$ have the minimum Conley-Zehnder indices/maximal cohomological grading, they represent classes in  $SH_+^{n-1,<D}(W;\Z)$, denoted by $[\check{\gamma}_1], \ldots, [\check{\gamma}_N]$.

In the following, we fix a contact form $\alpha_0$. For every $D > 0$, there exists a contact form $\alpha_D<\alpha_0$ such that Proposition \ref{prop:reeb} holds for the period threshold $D$. Let $M_D$ denote the cobordism from $\alpha_D$ to $\alpha_0$ in the symplectization of $(Y,\alpha_0)$. Then $W_D\cup M_D$ is a strict exact filling of $(Y,\alpha_0)$ that is homotopic to $W$, which is independent of $D$. Then by \eqref{eqn:viterbo},  we have a transfer map
$$SH_+^{*,<D}(W_D\cup M_D;\Z) \to SH_+^{*,<D}(W_D;\Z)$$
which is compatible with the connecting map to $H^*(W(\simeq W_D\simeq W_D\cup M_D);\Z)$. We will stretch on the contact boundary of $W_D$, the following propositions hold if we stretch the almost complex structure sufficiently.

\begin{proposition}\label{prop:n}
    Let $W^{flex}$ be the flexible filling of $Y$ and $W$ a topologically simple exact filling. For $D\gg 0$ and a sufficiently stretched almost complex structure, we have a commutative diagram
    $$\xymatrix{
     \la \check{\gamma}_1, \ldots, \check{\gamma}_N \ra \ar@{->>}[r] \ar[d]^{\Id} & SH^{n-1, <D}_+(W_D;\Z) \ar@{->>}[r] \ar[d]^{\simeq} & H^n(W_D;\Z)= H^n(W;\Z) \ar[d]^{\simeq} \\
     \la \check{\gamma}_1, \ldots, \check{\gamma}_N \ra \ar@{->>}[r] & SH^{n-1, <D}_+(W^{flex}_D;\Z) \ar@{->>}[r] & H^n(W^{flex}_D;\Z)= H^n(W^{flex};\Z) 
       }$$
       where the last two columns are isomorphisms that we will not specify.
\end{proposition}
\begin{proof}
By \cite[Corollary B]{filling}, we have $SH_+^{n-1}(W_D\cup M_D;\Z) \to H^n(W_D\cup M_D;\Z)=H^n(W;\Z)$ is an isomorphism. Therefore for $D\gg 0$, we have both
$$SH_+^{n-1,<D}(W_D\cup M_D;\Z) \to H^n(W_D\cup M_D;\Z),$$
$$SH_+^{n-1,<D}(W^{flex}_D\cup M_D;\Z) \to H^n(W^{flex}_D\cup M_D;\Z)$$
are surjective. Therefore by \eqref{eqn:viterbo}, we have both
$$SH_+^{n-1,<D}(W_D;\Z) \to H^n(W_D;\Z),\quad SH_+^{n-1,<D}(W^{flex}_D;\Z) \to H^n(W^{flex}_D;\Z)$$
are surjective. Following \cite[Proposition 3.12]{filling},  the identification between $SH_+^{n-1,<D}(W_D;\Z)$ with $SH_+^{n-1,<D}(W^{flex}_D;\Z)$ is by the obvious identification on generators and using sufficiently stretched almost complex structures which coincide in the cylindrical end for $W_D$ and $W^{flex}_D$. Therefore the first square is commutative. Finally since the identification of $H^n(W_D;\Z)$ with $H^n(W^{flex}_D;\Z)$ from \cite[Corollary B]{filling} is from the identification between $SH_+^{n-1}(W_D;\Z)$ and $SH_+^{n-1}(W^{flex}_D;\Z)$ using \cite[Proposition 3.12]{filling} and isomorphisms $SH_+^{n-1}(W_D;\Z)\to H^n(W_D;\Z)$,  $SH_+^{n-1}(W^{flex}_D;\Z)\to H^n(W^{flex}_D;\Z)$, the second square is commutative by construction.
\end{proof}

In the following, we will use $\alpha,\beta,\gamma$ to stand for Reeb orbits and $\hat{\alpha},\check{\alpha},\overline{\alpha}$ to stand for Hamiltonian orbits, where $\overline{\alpha}$ means that we do not specify whether it is a check or a hat orbit. 
\begin{proposition}\label{prop:id}
    Under the same assumptions in \Cref{prop:n}, for $D\gg 0$ and a sufficiently stretched almost complex structure,  there is a linear combination of Hamiltonian orbits $\sum a_i\overline{\alpha}_i$ of Conley-Zehnder index $n+1$, such that $\sum a_i\overline{\alpha}_i$  represents a class both in $SH^{-1,<D}_+(W_D;\Z)$ and  $SH^{-1,<D}_+(W^{flex}_D;\Z)$ and is sent to $1$ in both $H^0(W_D;\Z)$ and $H^0(W_D^{flex};\Z)$ respectively.
\end{proposition}
\begin{proof}
    By \cite[Corollary B]{filling}, we have a commutative diagram of isomorphisms
    $$\xymatrix{
    SH_+^*(W_D;\Z) \ar[r]^{\simeq}\ar[d]^{\simeq} & H^{*+1}(W_D;\Z)\ar[d]^{\simeq}\\
    SH_+^*(W^{flex}_D;\Z) \ar[r]^{\simeq} & H^{*+1}(W^{flex}_D;\Z)}
    $$
    For $D\gg 0$, \eqref{eqn:viterbo} implies that the following commutative diagram
    $$\xymatrix{
    SH_+^{-1,<D}(W_D;\Z) \ar@{->>}[r]\ar[d]^{\simeq} & H^{0}(W_D;\Z)\ar[d]^{\simeq}\\
    SH_+^{-1,<D}(W^{flex}_D;\Z) \ar@{->>}[r] & H^{0}(W^{flex}_D;\Z)}
    $$
    Since the first identification is from the obvious identification of generators (non-constant orbits), there is a linear combination of Hamiltonian orbits $\sum a_i\overline{\alpha}_i$ of Conley-Zehnder index $n+1$, such that $\sum a_i\overline{\alpha}_i$  represents a class both in $SH^{-1,<D}_+(W_D;\Z)$ and  $SH^{-1,<D}_+(W^{flex}_D;\Z)$ related by the map at the first column and is sent to $1$ in both $H^0(W_D;\Z)$ and $H^0(W_D^{flex};\Z)$ respectively.
\end{proof}

In the following, for the simplicity of notation, we will assume $\sum a_i \overline{\alpha}_i$ is represented by a single Hamiltonian orbit $\overline{\alpha}$.  The argument below works for linear combinations as long as they represent a closed class in the positive cochain complex.

In view of \Cref{prop:simple}, \cite[Corollary B]{filling} and the universal coefficient theorem, we have $H_*(W;\Z)=H_*(W^{flex};\Z)=\Z^m$ for the hypothetical filling $W$ of $Y$ in \Cref{thm:main}. In the following, we pick two sequences of smooth cycles $\{A_i\}_{1\le i \le m},\{B_i\}_{1\le i \le m}$ of $W$ as follows.
\begin{enumerate}
    \item $A_i=\sum_{j\in J_i} a_{i,j}\sigma^A_{i,j}$ for $a_{i,j}\in \Z$, where $\sigma^A_{i,j}$ is a smooth map from the $n$-simplex $\Delta^n$ to the interior of $W_D$, such that $\partial A_i=0$. The corresponding homology classes $\{[A_i]\}_{1\le i \le m}$ form a basic of $H_n(W;\Z)$. 
    \item Similarly, we have $B_i=\sum_{j\in J'_i} b_{i,j}\sigma^B_{i,j}$ such that $[A_i]=[B_i]\in H_n(W;\Z)$. We require that $\sigma^A_{i,j},\sigma^B_{k,l}$ intersect transversely (all intersections are contained in the interior of $\Delta^n$) for any $i,j,k,l$. We use $\sigma^A_{i,j}\cdot\sigma^B_{k,l}\in \Z$ to denote the transversal intersection number, i.e.\
    $$\sigma^A_{i,j}\cdot\sigma^B_{k,l} = \sum_{\substack{x,y\in \Delta^n \\ \sigma^A_{i,j}(x)=\sigma^B_{k,l}(y)}} \mathrm{sign}(x,y)$$
    where $\sign(x,y)$ the difference of orientations between $\mathrm{im}(\rD\sigma^A_{i,j}(x))\oplus \mathrm{im}(\rD\sigma^B_{k,l}(y))$ and $T_{\sigma_{i,j}^A(x)}W$. 
\end{enumerate}
The existence of such smooth cycles follows from \cite[Theorem 18.7]{zbMATH07848732} and the classical transversality theory. Now we pick $A=A_i,B=B_j$ and write $A=\sum_{i\in I} a_{i}\sigma_{i}^A,B=\sum_{j\in J}b_j\sigma_j^B$ for simplicity. 
\begin{lemma}\label{lem:intersection}
     $[A]\cdot [B]=\sum_{i\in I,j\in J} a_ib_j\sigma^A_{i}\cdot\sigma^B_{j}.$
\end{lemma}
\begin{proof}
   The argument is motivated by \cite[Theorem 11.10]{zbMATH00425858}. As $\sigma_i^A$ is transverse to $\sigma_j^B$, after small perturbations of $\sigma_i^A,\sigma_j^B$ in the interior of $\Delta^n$ which does not change $\sigma_i^A\cdot \sigma^B_j$,  we may assume for every $z\in W$ there exists at most one pair of indices $(i,j)$ and one pair of points $(x,y)$ such that $\sigma_i^A(x)=\sigma_j^B(y)=z$. Let $\{z_k\}_{1\le k \le M}$ be the set of such intersection points. By transversality at the intersection point, we can find small balls $\cB_k$ centered around $z_k$, such that the image of $\sigma_i^A,\sigma_j^B$ in $B_k$ are two $n$-dimensional coordinate subspaces in $\cB_k$ that intersect transversely. We can find neighborhoods $\cU_A,\cU_B$ with smooth boundary containing $\cup_i \Ima(\sigma_i^A),\cup_j\Ima(\sigma_j^B)$ respectively, such that $\cU_A\cap \cU_B=\cup_{1\le k \le M} \cB_k$. Let $D_W^A:H_*(\cU_A)\to H^{2n-*}(W,W-\cU_A), D_W^B:H_*(\cU_B)\to H^{2n-*}(W,W-\cU_B)$ be the duality isomorphism. Following the definition in \cite[\S 11]{zbMATH00425858}, the intersection number $[A]\cdot [B]$ is defined as $D_W^B([B])\cup D_W^A([A])\in H^{2n}(W,W-\cup_{1\le k\le M}\cB_k)$ composed with $H^{2n}(W,W-\cup_{1\le k\le M}\cB_k)\to H_0(\cup_{1\le k\le M}\cB_k)\to H_0(W)=\Z$, where the first map is the duality isomorphism. As it suffices to verify the relation over $\R$, we use the de Rham model for the cohomology. Assume $\eta_A\in H^n_c(\cU_A)$ is the Lefschetz dual of $[A]$, i.e.\ for any $\alpha\in H^n(\cU_A)$, we have $\int_A \eta :=\sum_{i}\int_{\Delta^n}(\sigma_i^A)^*\alpha=\int_{\cU_A}\alpha \wedge \eta_A$. Next, we focus on $\cB_k$ and write it as $D_1\times D_2$, where the ball $D_1\times \{0\}$ is the image of $\sigma_i^A$ and the ball $\{0\} \times D_2$ is the image of $\sigma^B_j$. Then $\eta_A$ is a differential form on $D_1\times D_2$ compactly supported in the $D_2$-fiber direction. Then by choosing $\alpha=\pi_1^*\nu$, where $\nu$ is compactly supported in $D_1$ with $\int \nu =1$ and $\pi_1$ is the projection $D_1\times D_2\to D_1$, we conclude that $\int_{\{*\} \times D^2}\eta_A=\pm a_i$ from $\int_A \eta=\int_{\cU_A}\alpha \wedge \eta_A$, the sign is determined if $T_0D_1\oplus T_0D_2$ coincides with $T_x\cB_k$ on orientation. After a modification by an exact form, we may assume $\eta_A$ in $\cB_k$ is in the form of $\pm \pi_2^*\nu$, where $\nu$ is the compactly supported form in $D_2$ as above and $\pi_2:D_1\times D_2\to D_2$ is the projection. As $D_W^B([B])\cup D_W^A([A])=\eta_B\cup \eta_A$, and in each $\cB_k$, $\eta_B\cup \eta_A=\pm a_ib_j\pi_1^*\nu \wedge \pi_2^*\nu$, whose total integration is $\sum_{i\in I,j\in J} a_ib_j\sigma^A_{i}\cdot\sigma^B_{j}$.
\end{proof}

For the periodic orbit $\overline{\alpha}$, we consider the following compactified moduli space:
$$\cM^{\overline{\alpha}}_{\sigma^A_i,\sigma^B_j}:=\overline{\left\{ u:\C \to \widehat{W} \left| (\rd u - v)^{0,1}=0, u(\infty) = \overline{\alpha}\right.\right\}\prescript{}{ev_0\times ev_1}{\times}_{\sigma^A_i\times \sigma^B_j}(\Delta^n\times \Delta^n)}.$$
Here we have
\begin{enumerate}
    \item   $v = X_{H} \otimes \beta$ with $H$ a Hamiltonian as before and $\beta$ a one form, such that  $\beta = \rd t$ near the ends w.r.t.\ fixed cylindrical coordinates (we fix biholomorphisms from $ (-\infty,0)\times S^1_t$ to neighborhoods of $0,1$,  and from $(0,+\infty)\times S^1_t$ to a neighborhood of $\infty$, i.e.\ $\infty$ is a positive puncture, while $0,1$ are negative punctures)  and $\rd \beta\le 0$.
    \item  $u(\infty) = \overline{\alpha}$ is a short hand for $\lim_{s\to \infty} u =\overline{\alpha}$ for the cylindrical coordinate $(0,+\infty)_s\times S^1_t \to \C, (s,t)\mapsto e^{2\pi(s+\mathbf{i}t)}$.
    \item $ev_0,ev_1$ are evaluation maps of the curve $u$ at $0,1$. Since $H=0$ on the image of $A,B$, the removal of singularity implies that $u$ can be viewed as a map on $\C$.
    \item\label{trans} The transversality assumptions are: both
    $$\cM:=\left\{ u:\C \to \widehat{W} \left| (\rd u - v)^{0,1}=0, u(\infty) = \overline{\alpha}, u(0),u(1)\in W^{\circ}_D(\text{the interior})\right.\right\}$$
    and the fiber product $\prescript{}{ev_0\times ev_1}{\times}_{\sigma^A_i\times \sigma^B_j}$ are cut out transversely on the interior as well as each boundary and corner strata of $\Delta^n\times \Delta^n$. Alternatively, we require the induced Fredholm section on the fiber product of the Banach manifold defining $\cM$ with $\Delta^n\times \Delta^n$ by $\prescript{}{ev_0\times ev_1}{\times}_{\sigma^A_i\times \sigma^B_j}$ is cut out transversely on each strata. The former transversality property will imply the latter property, but the latter Fredholm problem is again index $1$ (on the top stratum) and it is easy to see that it can be arranged by the somewhere injective property of $u$. However, the first seemingly stronger transversality can also be arranged in view of the submersion property of the universal moduli spaces \cite[Proposition 3.4.2, Lemma 3.4.3]{zbMATH06054087}.
    \item Such moduli spaces have a coherent orientation, which takes the orientations of $\cM$, $W_D$, $\Delta^n$ as well as $ev_0\times ev_1$ and $\sigma^A_i\times \sigma^B_j$ into account.
\end{enumerate}
Similarly for another orbit $\overline{\beta}$, we can define $\cM^{\overline{\alpha}}_{\overline{\beta}, \sigma_j^B}$ and $\cM^{\overline{\alpha}}_{\sigma^A_i, \overline{\beta}}$ as follows.
$$\cM^{\overline{\alpha}}_{\overline{\beta}, \sigma^B_j}:=\overline{\left\{ u:\C \to \widehat{W} \left| (\rd u - v)^{0,1}=0, u(\infty) = \overline{\alpha}, u(0)=\overline{\beta}\right.\right\}\prescript{}{ev_1}{\times}_{\sigma_j^B}\Delta^n}$$
$$\cM^{\overline{\alpha}}_{\sigma_i^A, \overline{\beta}}:=\overline{\left\{ u:\C \to \widehat{W} \left| (\rd u - v)^{0,1}=0, u(\infty) = \overline{\alpha}, u(1)=\overline{\beta}\right.\right\}\prescript{}{ev_0}{\times}_{\sigma_i^A}\Delta^n}$$
where $u(0),u(1)=\overline{\beta}$ is the short hand for $\lim_{s\to-\infty}u=\overline{\beta}$ using the cylindrical coordinates near negative punctures $0$ and $1$ respectively. We also define $\cM^{\overline{\alpha}}_{\sigma^A_i}$ to be the compactification of the following.
$$\left(\left\{u:\C \to \widehat{W}\left| (\rd u- X_{H}\rd t)^{0,1}=0, u(\infty) = \overline{\alpha}\right.\right\}/\R \right)\prescript{}{ev_0}{\times}_{\sigma_i^A}\Delta^n.$$
The transversality requirement regarding boundary and corners of $\Delta^n$ for those moduli spaces is identical to \eqref{trans} above. 

For $\overline{\alpha}$ from \Cref{prop:id} and $\check{\gamma}$ from \Cref{prop:n}, the dimension of $\cM^{\overline{\alpha}}_{\sigma^A_i,\sigma^B_j}$ is $1$ and the dimensions  of $\cM^{\overline{\alpha}}_{\sigma^A_{i},\check{\gamma}},\cM^{\overline{\alpha}}_{\check{\gamma},\sigma^B_j},\cM^{\check{\gamma}}_{\sigma^A_i},\cM^{\check{\gamma}}_{\sigma^B_j}$ are $0$. The transversality requirement from \eqref{trans} implies that the latter four moduli spaces sit over the interior of $\Delta^n$, while part of the boundary of $\cM^{\overline{\alpha}}_{\sigma^A_i,\sigma^B_j}$ is from $\partial \sigma_i^A$ and $\partial \sigma_j^B$ sitting over the interior of faces of $\Delta^n$. Then we define the following counting:
\begin{enumerate}
    \item $\#\cM^{\overline{\alpha}}_{A,\check{\gamma}}:=\sum_{i\in I}a_i\#\cM^{\overline{\alpha}}_{\sigma^A_{i},\check{\gamma}}$ and  $\#\cM^{\overline{\alpha}}_{\check{\gamma},B}:=\sum_{j\in J}b_j\#\cM^{\overline{\alpha}}_{\check{\gamma},\sigma^B_j}$.
    \item $\#\cM^{\check{\gamma}}_{A}:=\sum_{i\in I}a_i\#\cM^{\check{\gamma}}_{\sigma^A_i}$. 
\end{enumerate}

\begin{proposition}\label{prop:int}
	    Let $\overline{\alpha}$ be the class in Proposition \ref{prop:id} and $A,B$ above , then the intersection number is
        \begin{equation}\label{eqn:intersection}
            [A]\cdot [B] = \sum_{i=1}^N\left(\#\cM^{\overline{\alpha}}_{\check{\gamma}_i,B}\times \#\cM^{\check{\gamma}_i}_{A}+\#\cM^{\overline{\alpha}}_{A,\check{\gamma}_i}\times \#\cM^{\check{\gamma}_i}_{B} \right).
        \end{equation}
\end{proposition}
\begin{proof}
It follows from the boundary configuration of $\cM^{\overline{\alpha}}_{\sigma^A_i,\sigma^B_j}$ whose dimension is $1$, i.e.\ from the counting 
\begin{equation}\label{eqn:boundary}
    \sum_{i\in I, j\in J}  a_ib_j\#\partial \cM^{\overline{\alpha}}_{\sigma^A_i,\sigma^B_j}=0.
\end{equation}
Note that $\dim \cM^{\overline{\beta}}_{\sigma^A_i} = \mu_{\CZ}(\overline{\beta})-1$, and all periodic orbits have Conley-Zehnder indices greater than $1$ unless they are one of $\check{\gamma}_i$. By degree reason, the Floer type breakings near negative punctures $0,1$ in \eqref{eqn:boundary} give rise to the right-hand side of \eqref{eqn:intersection}. Those boundaries in \eqref{eqn:boundary} coming from $\partial \sigma_i^A,\partial \sigma^B_j$ cancel with each other as $\partial \sum a_i\sigma_i^A=\partial \sum b_j \sigma^B_j=0$.
    \begin{figure}[H] {\small
    \begin{overpic}[scale=0.5]
    {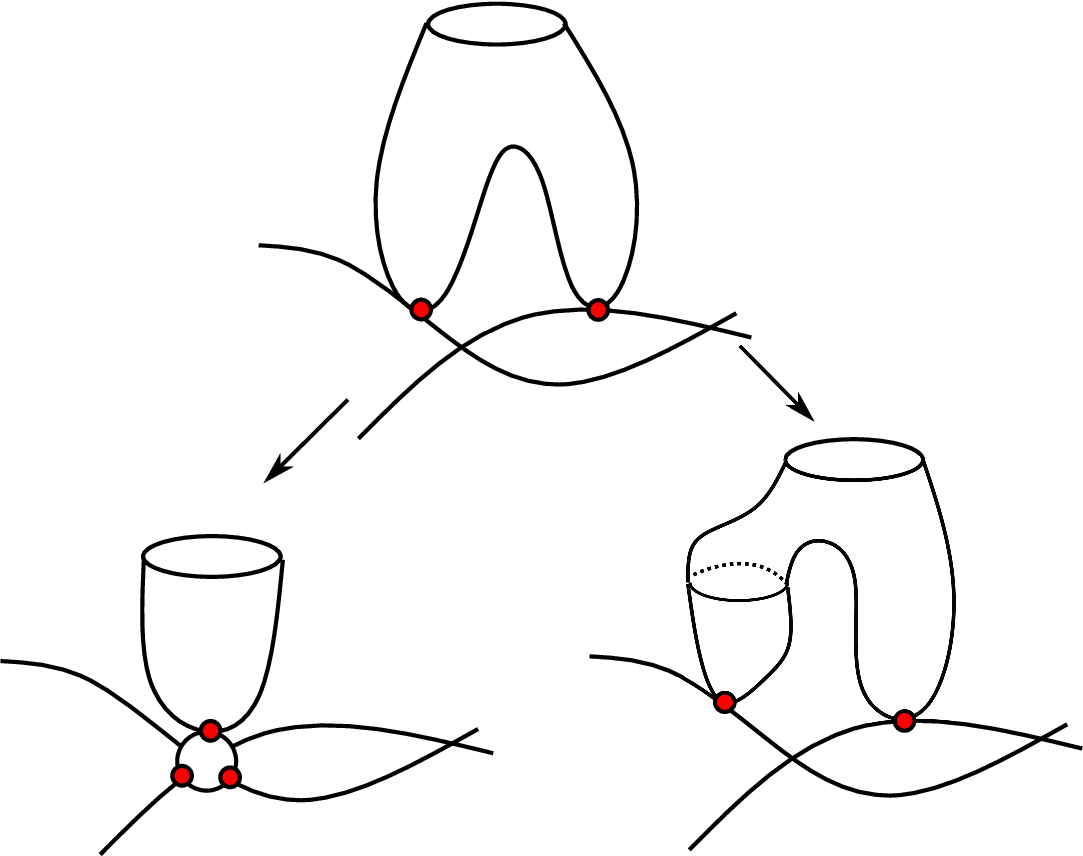}
    \put(45,80) {$\overline{\alpha}$}
    \put(25,58) {$\sigma^A_i$}
    \put(35,35) {$\sigma^B_j$}
    \put(15,2) {Constant sphere at $\sigma^A_i\pitchfork \sigma^B_j$}
    \put(68,25) {$\overline{\beta}$}
    \put(8,40) {degeneration}
    \put(72,45) {degeneration}
    \end{overpic}}
    \caption{Degeneration of $\cM^{\overline{\alpha}}_{\sigma^A_i,\sigma^B_j}$}
    \label{fig:deg}
    \end{figure}
Since $\overline{\alpha}$ is closed in the positive symplectic cohomology, the Floer type breakings near $\infty$ at a non-constant orbit sum up to $0$.  We consider the Floer type breakings near $\infty$ at the interior of $W$. By the integrated maximal principle, the bottom curve is contained in $W$, where the equation is the Cauchy-Riemann equation. By the exactness of $W$, such a curve must be constant with three punctures resting on a point in $\sigma_i^A\pitchfork \sigma_j^B$. Therefore such degeneration can be identified with curves $u:\C \to \widehat{W}$ solves $(\rd u-X_{H}\rd t)^{0,1}=0$ with $u(\infty)=\overline{\alpha}$ and $u(0)\in \sigma^A_i\pitchfork \sigma^B_j$ modulo the $\R$ translation. Since $\overline{\alpha}$ is mapped to $1 \in H^0(W;\Z)$, i.e.\ the count of curves $u$ with a point constraint at $u(0)$ and $u(\infty)=\overline{\alpha}$ modulo $\R$ is $1$ when transversality holds. Therefore, this type of degeneration is counted as the transversal intersection number $\sigma_i^A\cdot \sigma_j^B$. Therefore those types of degeneration in \eqref{eqn:boundary} sum to $[A]\cdot [B]$ by \Cref{lem:intersection}. Since the orientation problem at hand is identical to the case of continuation maps, this explains the sign difference between breaking at the positive puncture and negative punctures, hence the relation \eqref{eqn:intersection}. The transversality (requirement \eqref{trans} above) of those moduli spaces holds if we choose $J$ generically, following the same argument as \cite[Proposition 2.8]{filling} using the submersion property of the universal moduli spaces \cite[Proposition 3.4.2, Lemma 3.4.3]{zbMATH06054087}.
\end{proof}

When $A,B$ are $n$-dimensional closed submanifolds of $W$ such that $A\pitchfork B$, we can define the compactified moduli space 
\begin{equation}\label{eqn:moduli1}
    \cM^{\overline{\alpha}}_{A,B}:=\overline{\left\{ u:\C \to \widehat{W} \left| (\rd u - v)^{0,1}=0, u(\infty) = \overline{\alpha}, u(0) \in A, u(1) \in B\right.\right\}}
\end{equation}
whose boundary configuration gives the proof of \Cref{prop:int}. Here we define
\begin{equation}\label{eqn:moduli2}
    \cM^{\overline{\alpha}}_{\overline{\beta}, B}:=\overline{\left\{ u:\C \to \widehat{W} \left| (\rd u - v)^{0,1}=0, u(\infty) = \overline{\alpha}, u(0)=\overline{\beta}, u(1)\in B\right.\right\}},
\end{equation}
\begin{equation}\label{eqn:moduli3}
    \cM^{\overline{\alpha}}_{A, \overline{\beta}}:=\overline{\left\{ u:\C \to \widehat{W} \left| (\rd u - v)^{0,1}=0, u(\infty) = \overline{\alpha}, u(0)\in A, u(1)=\overline{\beta}\right.\right\}},
\end{equation}
\begin{equation}\label{eqn:moduli4}
    \cM^{\overline{\alpha}}_{A}=\overline{\left\{u:\C \to \widehat{W}\left| (\rd u- X_{H}\rd t)^{0,1}=0, u(\infty) = \overline{\alpha}, u(0)\in A\right.\right\}/\R}.
\end{equation}
The right-hand side of \eqref{eqn:intersection} is the counting of the moduli spaces above. It coincides with the definition before \Cref{prop:int} after we choose appropriate triangulations on $A$ and $B$. The proof of \Cref{prop:int} illustrates how to work with smooth cycles $\{A_i,B_i\}_{1\le i \le m}$ and 
shows that there is no essential difference (except for introducing cumbersome notation) from working with the simpler moduli spaces \eqref{eqn:moduli1}-\eqref{eqn:moduli4} above when $A_i,B_i$ are represented by closed submanifolds. In the following, we will assume $\{A_i,B_i\}_{1\le i \le m}$ are represented by closed $n$-submanifolds for simplicity and work with \eqref{eqn:moduli1}-\eqref{eqn:moduli4}. In general, we need to work with the version with smooth cycles and argue as in \Cref{prop:int}.

\begin{proposition}\label{prop:pair}
	For a sufficiently stretched almost complex structure, we have
	$$ \#\cM^{\check{\gamma}_i}_B=\la \delta([\check{\gamma}_j]), [B] \ra,$$
	where the last pairing is the natural map $H^n(W;\Z)\otimes H_n(W;\Z) \to \Z$ and $\delta$ is the map $SH^*_+(W;\Z)\to H^{*+1}(W;\Z)$.
\end{proposition}
\begin{proof}
	Given a Morse function $g$ on $W$ such that $\partial_r g >0$ on $\partial W$, then we can represent a cochain complex of $H^*(W;\Z)$ by critical points of $g$, then the paring of a critical point $x$ with a cycle $B$ is the intersection number of the ascending manifold of $x$ with $B$, i.e.\ the top right corner of \Cref{fig:def1}, transversality can be achieved by either using a generic metric define $\nabla g$ or perturbing $B$. Following \cite[\S 2]{filling}, the map $\delta$ can be represented by counting the moduli space of $(u,l)$ with $u$ solves the Floer equation and $u(0)\in W$ and $l$ is a gradient trajectory from $u(0)$ to a critical point $x$, i.e.\ top left corner of \Cref{fig:def1}.
    \begin{figure}[H] {\small
    \begin{overpic}
    {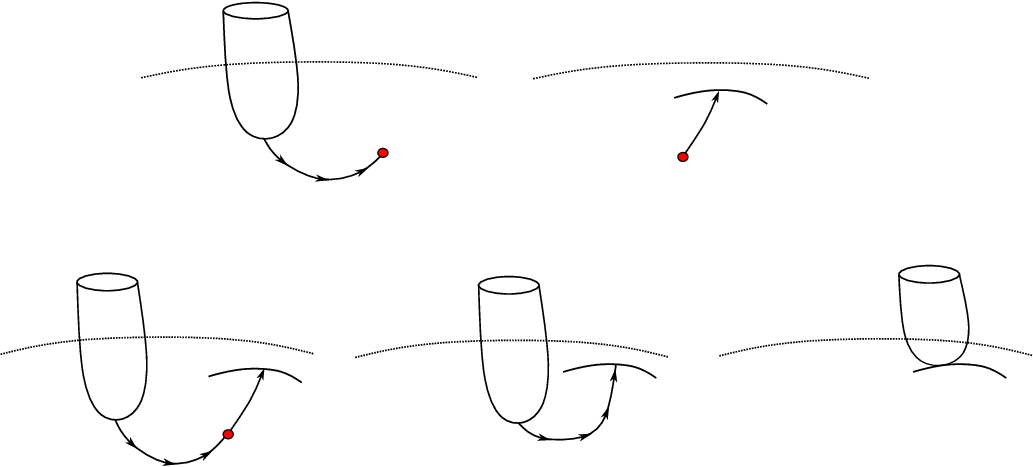}
    \put(38,30) {$x$}
    \put(12,38) {$\partial W$}
    \put(30,25) {$ \delta([\check{\gamma}_j])$}
    \put(55,25) {Paring the Morse cocohain complex with $B$}
    \put(75,35) {$B$}
    \put(15,-2) {$\la \delta([\check{\gamma}_j]), B \ra$}
    \put(90,6) {$\cM_{\check{\gamma}_i,B}$}
    \put(25,35) {$u$}
    \put(33,29) {$l$}
    \put(15,3) {$l_1$}
    \put(25,5) {$l_2$}
    \put(50,1) {finite gradient flow of $\nabla g$}
    \end{overpic}}
    \caption{$\la \delta([\check{\gamma}_j]), B \ra$ and its deformation}
    \label{fig:def1}
    \end{figure}
    \noindent
    Therefore $\la \delta([\check{\gamma}_j]), B \ra$ counts the moduli space of $(u,l_1,l_2)$ with $l_1,l_2$ be two half infinite gradient trajectories connected at an index $n$ critical point of $g$, i.e.\ the bottom left corner of \Cref{fig:def1}. Then by shrinking the time of the gradient flow lines from $\infty$ to $0$ as in \cite[\S 3.1]{filling} illustrated in the bottom row of \Cref{fig:def1}, and note that $[\check{\gamma}_i]$ is closed in positive symplectic cohomology and $B$ is closed, the count equals to time length $0$ count, which is $\#\cM^{\check{\gamma}_i}_B$.
\end{proof}

To compute $\cM^{\overline{\alpha}}_{\check{\gamma}_j,B}$, we will perform a full neck-stretching on the boundary of $W_D$. For this, we first define the relevant moduli spaces. Let $\widehat{Y}$ denote the symplectization $( (0, \infty)_r\times Y, r\alpha_D)$, which is equipped with a Hamiltonian $H$ such that $H = 0$ on $ (0, 1)\times Y$ and after that it is the same as the $H$ on $\widehat{W}_D$ as in the beginning of \S \ref{s4}. Then we define $\cN^{\overline{\alpha}}_{\check{\gamma}_i,\gamma_j}$ to be the compactification of the
following moduli space
$$\left\{ u:\CP^1\backslash \{\infty,0,1\} \to \widehat{Y}\left| (\rd u-X_H\otimes \beta)^{0,1}=0, u(\infty) = \overline{\alpha}, u(0)=\check{\gamma}_i, u(1) = (0,\gamma_j)\right.\right\}$$
i.e. $u(\infty),u(0)$ are asymptotic to Hamiltonian orbits and $u(1)$ is asymptotic to a Reeb orbit at a negative puncture. We define $\cN^{\gamma_j}_{B}$ to be the compactification of the following moduli space.
$$\left\{u:\C \to \widehat{W}\left| (\rd u)^{0,1}=0, u(\infty) = (+\infty, \gamma_j), u(0)\in B\right.\right\}/ \R\times S^1$$
    \begin{figure}[H] {\small
    \begin{overpic}[scale=0.7]
    {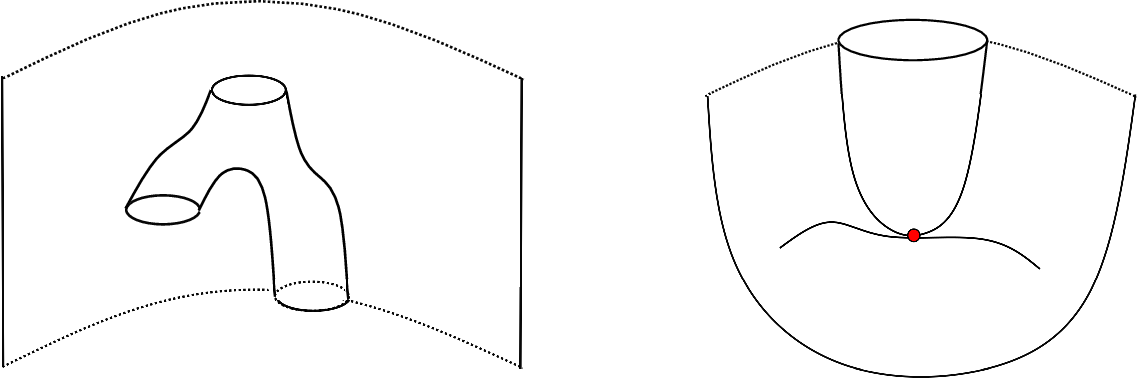}
    \put(20,30) {$Y\times \{\infty\}$}
    \put(17,3) {$Y\times \{0\}$}
    \put(21.5,24.5) {$\overline{\alpha}$}
    \put(13,11) {$\check{\gamma}_i$}
    \put(26,10) {$\gamma_j$}
    \put(80,29) {$\overline{\alpha}$}
    \put(70,10) {$B$}
    \put (90,30) {$Y\times \{\infty\}$}
    \put (90,14) {$\widehat{W}_D$}
    \end{overpic}}
    \caption{$\cN^{\overline{\alpha}}_{\check{\gamma}_i,\gamma_j}$ and $\cN^{\gamma_j}_{B}$}
    \label{fig:n}
    \end{figure}

\begin{proposition}\label{prop:SFT1}
	For a sufficiently stretched almost complex structure, we have
	$$\# \cM^{\overline{\alpha}}_{\check{\gamma}_j,B} = \sum_{k=1}^N \#\left(\cN^{\overline{\alpha}}_{\check{\gamma}_j, \gamma_k} \times \cN^{\gamma_k}_{B}\right).$$
    Similarly, we have 
    $$\# \cM^{\overline{\alpha}}_{A,\check{\gamma}_j} = \sum_{k=1}^N \#\left(\cN^{\overline{\alpha}}_{\gamma_k,\check{\gamma}_j} \times \cN^{\gamma_k}_{A}\right).$$
\end{proposition}
\begin{proof}
We perform a full neck-stretching along the boundary of $W_D$, then any curve in $\cM^{\overline{\alpha}}_{\check{\gamma}_j, B}$ will converge to a SFT-building type curve, since $B \subset W_D$. The top-level curve is necessarily connected by \cite[Proposition 9.17]{cieliebak2018symplectic}, with one fixed negative puncture at $1$ which will connect to the component that eventually intersects $B$. But there might be other free-moving punctures that will eventually be closed off by holomorphic planes. Let $\gamma$ denote the Reeb orbit on the puncture $1$, and $\beta_i$ be those Reeb orbits on those free punctures. Then the virtual dimension of this moduli space of the topmost curve is $\mu_{\CZ}(\overline{\alpha})-\mu_{\CZ}(\check{\gamma}_j)-(\mu_{\CZ}(\gamma)+n-1)-\sum (\mu_{\CZ}(\beta_i)+n-3)$. Although there might be Reeb orbits of $(Y,\alpha_D)$ with $\mu_{\CZ}<1$, all Reeb orbits that can potentially appear in $\{\gamma,\beta_1,\beta_2,\ldots\}$ must have $\mu_{\CZ} \ge 1$ by \Cref{prop:reeb} as their period must be smaller than that of $\alpha$ which is at most $D$. Since we can assume transversely for the topmost curve, as we are using $S^1$-dependent almost complex structures on the topmost level. Therefore the only possibility for  $\gamma$ is one of $\gamma_i$ and there is no $\beta_i$, and we have the expected dimension is $0$, for otherwise, the expected dimension is negative. After the topmost level, we might have several levels of curves in the symplectization, with the topmost curve in the symplectization portion with only one positive puncture asymptotic to $\gamma_i$. Since $\gamma_i$ is simple, the topmost curve in the symplectization portion is necessarily somewhere injective, hence we can assume transversality for the topmost curve. Since the curve must connect to some component that eventually intersects $B$, therefore the curve must have at least one negative end $\gamma'$, then the expected dimension of the moduli space of this curve is $\mu_{\CZ}(\gamma_i)-\mu_{\CZ}(\gamma')-\sum_j (\mu_{\CZ}(\beta_j)+n-3)-1$. Since $\mu_{\CZ}(\gamma_i)$ is the lowest and all SFT grading $\mu_{\CZ}(\beta_i)+n-3$ are positive, we have the dimension is negative. As a result, there is no topmost curve in the symplectization, i.e.\ no curves in the symplectization portion. The last part is the curve in the completion $\widehat{W}_D$, which is exactly $\cN_{\gamma_i,B}$ with expected dimension $0$. Since $\gamma_i$ is simple, transversality is not an issue. Therefore the right-hand side is the count from the fully stretched almost complex structure. 
    \begin{figure}[H] {\small
    \begin{overpic}[ scale=0.6]
    {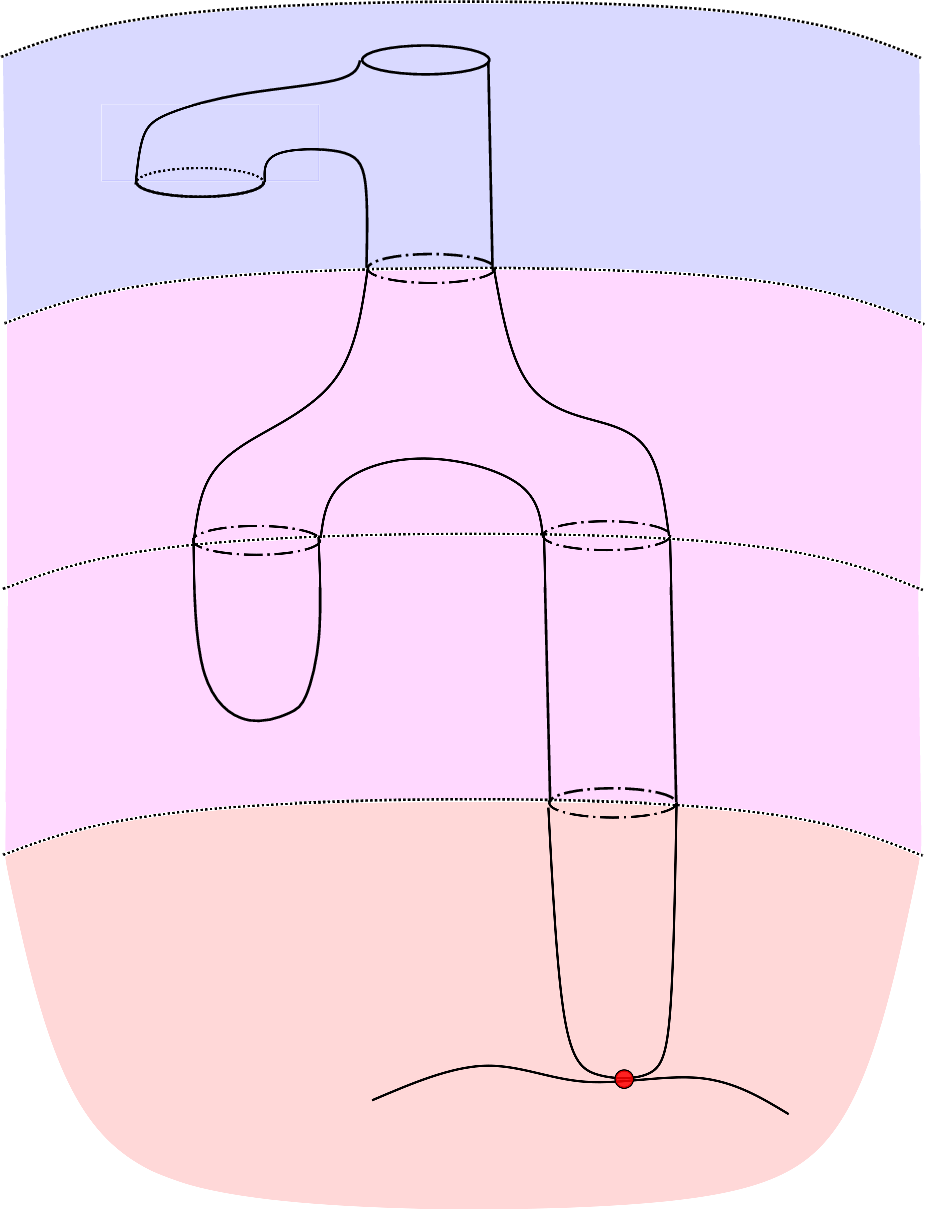}
    \put(35,97) {$\overline{\alpha}$}
    \put(15,80) {$\check{\gamma}_j$}
    \put(35,80) {$\gamma$}
    \put(40,10) {$B$}
    \put(57,40) {trivial cylinder}
    \put(10,20) {$\widehat{W}_D$}
    \put(10,40) {$\widehat{Y}$}
    \put(10,60) {$\widehat{Y}$}
    \put (42,90) {topmost curve}
    \put (48,70) {topmost curve in the} 
    \put (48,66) {symplectization portion}
    \end{overpic}}
    \caption{The fully stretched picture of curves in $\cM^{\overline{\alpha}}_{\check{\gamma}_j, B}$}
    \label{fig:SFT}
    \end{figure}
\noindent
If we assume we start with an almost complex structure that is stretched enough, we may assume in the process of stretching there is no curve in $\cM^{\overline{\alpha}}_{\overline{\beta}}$ and $\cM^{\overline{\beta}}_{\check{\gamma}_i}$ with expected dimension $-1$. Moreover, we also assume there is no curve in $\cM^{\overline{\beta}}_{\overline{\gamma},B}$ with expected dimension $-1$ in the process of stretching, for otherwise, we have a curve in a moduli space of negative dimension after the full stretch by compactness. Since $B$ is closed, in the process of neck-stretching, we only have $\cM^{\overline{\alpha}}_{\check{\gamma}_j,B}$ and $\sum_{k=1}^N \# \left(\cN^{\overline{\alpha}}_{\check{\gamma}_j, \gamma_k} \times \cN^{\gamma_k}_{B}\right)$ as boundary corresponding to the two ends of the neck-stretching parameter.
\end{proof}
We define $\cN^{\check{\gamma}_i}_{\gamma_j}$ to be the compactified moduli space of 
$$\left\{ u: \R \times S^1 \to \widehat{Y}\left| (\rd u-X_H\rd t)^{0,1}=0, u(\infty)=\check{\gamma}_i, u(-\infty)=(0,\gamma_j)\right.\right\}/\R$$
\begin{proposition}	\label{prop:SFT2}
	For a sufficiently stretched almost complex structure and $H$ sufficiently close to the autonomous one which only depends on $r$, we have 
	$$\sum_{i=1}^j\#\left(\cN^{\check{\gamma}_j}_{\gamma_i}\times \cN^{\gamma_i}_{ B}\right) = \la \delta([\check{\gamma}_j]), B \ra,$$
	and $\#\cN_{\check{\gamma}_j,\gamma_j}=1$.
\end{proposition}
\begin{proof}
	The proof is similar to Proposition \ref{prop:SFT1} by fully stretching the moduli space $\cM^{\check{\gamma}_j}_{B}$. By a similar dimension argument, the moduli space must break into $\cN^{\check{\gamma}_j}_{\gamma_i}\times \cN^{\gamma_i}_{B}$. Therefore it suffices to prove $i\le j$. When $H$ is autonomous and only depends on $r$, $X_H$ is parallel to the Reeb vector. Therefore for any solution $u \in \cN^{\check{\gamma}_j}_{\gamma_i}$, we have the $\alpha$-energy $\int u^*(\rd\alpha_D)\ge 0$ (as in \eqref{eqn:positive2}), which implies that the period of $\gamma_j$ must be greater than $\gamma_i$ unless $\gamma_i=\gamma_j$. Then for $H$ sufficient close to the autonomous one, we have $\cN_{\check{\gamma}_j,\gamma_i} \ne \emptyset$ implies that $i\le j$. Moreover, for the autonomous Hamiltonian, any curve $u$ in $\cN_{\check{\gamma}_j,\gamma_j}$ have trivial $\alpha$-energy, hence $u$ is contained in $\Ima(\gamma_j)\times \R \subset \widehat{Y}$. In this case, the equation becomes 
     $$\partial_s u+J(\partial_tu-h'(r)\partial_{\theta})=0, \quad u:\R_s\times S^1_t \to \R_{\rho}\times ([0,T]_{\theta}/{0\simeq T}),$$
     where $T$ is period of $\gamma_j$ so that $\theta$ is the coordinate on $\Ima(\gamma_j)$, $\rho=\log(r)$ on $\widehat{Y}=(0,\infty)_r\times Y$, $J(\partial_{\rho})=\partial_{\theta}$. It has a solution 
     $$u(s,t)=(f(s),Tt+\theta_0)$$
     where $f:\R_s\to \R_\rho$ is a solution to $f'=T-h'(e^f)$ such that $f=Ts+C$ for $s\ll 0$. Different solutions are related by translations on $\R_s$. Since the Floer-cylinder with zero $\alpha$-energy sitting over $\Ima(\gamma_j)$ is in the form of $u(s,t)=(Ts+C,Tt+\theta_0)$ for $s\ll 0$ in a cylindrical coordinate near the negative puncture (not necessarily the standard one used above), by the unique continuation of holomorphic curves (we use Gromov's graph trick to turn the Floer equation into a $\overline{\partial}$ equation), under a biholomorphism, $u(s,t)$ must be $u(s,t)=(f(s),Tt+\theta_0)$ above. In particular, the moduli space
     $$\left\{ u: \R \times S^1 \to \widehat{Y}\left| (\rd u-X_H\rd t)^{0,1}=0, u(\infty)\in \overline{\gamma}_j, u(-\infty)=(0,\gamma_j)\right.\right\}/\R$$
     is diffeomorphic to $S^1$ and is cut out transversely, similar to the case that a trivial cylinder in symplectization is cut out transversely. Therefore the cascades moduli space $\cN_{\check{\gamma}_j,\gamma_j}$ 
     $$\left\{ u: \R \times S^1 \to \widehat{Y}\left| (\rd u-X_H\rd t)^{0,1}=0, u(\infty)=\check{\gamma}_j\in \overline{\gamma}_j, u(-\infty)=(0,\gamma_j)\right.\right\}/\R$$
     is cut out transversely and consists of one point. By the same analysis in \cite[Main Theorem (iii)]{bourgeois2009symplectic}, the cascades moduli space $\cN_{\check{\gamma}_j,\gamma_j}$ has a one-to-one correspondence with the corresponding moduli space $\cN_{\check{\gamma}_j,\gamma_j}$ for a perturbed Hamiltonian. One can avoid such perturbation if one uses a cascades setup with an autonomous Hamiltonian as in \cite{MR2471597,bourgeois2009symplectic} for previous propositions.
\end{proof}

\begin{proof}[Proof of Theorem \ref{thm:main}]
    In view of \Cref{prop:simple}, $c^\Q_1(W)=0$ implies that $W$ is topologically simple. If we are given $\la \delta([\check{\gamma}_j]), B\ra$ and $\#\cN^{\check{\gamma}_i}_{\gamma_j}$, we can solve uniquely $\#\cN^{\gamma_i}_{B}$ by Proposition \ref{prop:SFT2},  since the coefficient matrix is triangular with ones on the diagonal.  Then by \Cref{prop:int,prop:pair,prop:SFT1}, we can represent the intersection $A\cdot B$ by $\cN^{\overline{\alpha}}_{\check{\gamma}_i,\gamma_j}$, $\cN^{\overline{\alpha}}_{\gamma_i,\check{\gamma}_j}$, $\cN^{\check{\gamma}_i}_{\gamma_j}$,  $\la \delta([\check{\gamma}_j]), A\ra$ and  $\la \delta([\check{\gamma}_j]), B\ra$. The first three moduli spaces are independent of the filling, as they are contained in the symplectization. Note that $H^n(W;\Z)$ is independent of filling, and a basis can be represented by combinations of $[\check{\gamma}_i]$, and the same combinations work for $W^{flex}$ by \Cref{prop:n}.  By the universal coefficient theorem, $H_n(W;\Z)$ is isomorphic to the free part of $H^n(W;\Z)$ since $H^*(W;\Z)$ is supported in degree $\le n$.  Fixing a basis of a fixed free part of $H^n(W;\Z)$ induces a dual basis on $H_n(W;\Z)$. We use this dual basis to identify the homology of two fillings, that is we have an isomorphism $H_n(W;\Z)\stackrel{\psi}{\to} H_{n}(W^{flex};\Z)$, such that 
    \begin{equation}\label{eqn:iso}
        \la [\check{\gamma}_i], A \ra = \la [\check{\gamma}_i], \psi(A)\ra, A \in H_n(W;\Z),
    \end{equation}
    where we view $[\check{\gamma}_i]$ both in $H^n(W;\Z)$ and $H^n(W^{flex};\Z)$.
    Since $\cN^{\check{\gamma}_i}_{\gamma_j}$ is contained in the symplectization and is the same for $W$ and $W^{flex}$, by \Cref{prop:SFT2} we have 
   \begin{equation}\label{eqn:iso'}
       \# \cN^{\gamma_i}_B=\# \cN^{\gamma_i}_{\psi(B)}.
   \end{equation}
    Therefore we have 
    \begin{eqnarray*}
        A\cdot B &\stackrel{\text{Prop }  \ref{prop:int}}{\eqdef}&\sum_{i=1}^N\left(\#\left(\cM^{\overline{\alpha}}_{\check{\gamma}_i,B}\times \cM^{\check{\gamma}_i}_{A}\right)+\# \left(\cM^{\overline{\alpha}}_{A,\check{\gamma}_i}\times \cM^{\check{\gamma}_i}_{B}\right) \right)\\
        &\stackrel{\text{Prop } \ref{prop:SFT1}}{\eqdef} &\sum_{i=1}^N\sum_{j=1}^N\left(\#\left(\cN^{\overline{\alpha}}_{\check{\gamma}_i, \gamma_j} \times \cN^{\gamma_j}_{B}\times \cM^{\check{\gamma}_i}_{A}\right)+\# \left(\cN^{\overline{\alpha}}_{\gamma_j,\check{\gamma}_i}\times \times \cN^{\gamma_j}_{A} \times \cM^{\check{\gamma}_i}_{B}\right) \right)\\
       &\stackrel{\text{Prop }  \ref{prop:pair}}{\eqdef} &\sum_{i=1}^N\sum_{j=1}^N\left(\#\left(\cN^{\overline{\alpha}}_{\check{\gamma}_i, \gamma_j} \times \cN^{\gamma_j}_{B}\right)\times \la [\check{\gamma}_i],A\ra+\# \left(\cN^{\overline{\alpha}}_{\gamma_j,\check{\gamma}_i}\times \times \cN^{\gamma_j}_{A} \right)\times \la [\check{\gamma}_i],B\ra \right)\\
       & \stackrel{\eqref{eqn:iso},\eqref{eqn:iso'}}{\eqdef}&\sum_{i=1}^N\sum_{j=1}^N\left(\#\left(\cN^{\overline{\alpha}}_{\check{\gamma}_i, \gamma_j} \times \cN^{\gamma_j}_{\psi(B)}\right)\times \la [\check{\gamma}_i],\psi(A)\ra+\# \left(\cN^{\overline{\alpha}}_{\gamma_j,\check{\gamma}_i}\times \times \cN^{\gamma_j}_{\psi(A)} \right)\times \la [\check{\gamma}_i],\psi(B)\ra \right)\\
       & \eqdef & \psi(A)\cdot \psi (B).
    \end{eqnarray*}
     That is $\psi$ identifies the intersection form.
\end{proof}	

\begin{remark}
A natural question is whether some of the above counts bear homological meaning, i.e.\ can they be phrased as structural maps in ($S^1$-equivariant) symplectic cohomology or symplectic field theory? 
\begin{enumerate}
    \item Under the identification of linearized contact homology and positive $S^1$-equivariant symplectic cohomology proved by Bourgeois and Oancea \cite{MR2471597}, the counting of $\cN_{\gamma_i,B}$ should be contained in the map $SH^*_{+,S^1}(W;R) \to H^{*+1}(W;R)\otimes_R R[u,u^{-1}]/u \to H^{*+1}(W)$.
    \item The counting of $\cN_{\check{\gamma}_i,\gamma_j}$ should be contained in the map $\iota:SH^*_{+}(W) \to SH^*_{+,S^1}(W)$ from the Gysin exact sequence \cite{zbMATH07004403}.
    \item The counting of $\cM_{\overline{\alpha},\check{\gamma}_j,\gamma_k}$ does not give rise to structural maps (but should be a piece of a structure, where we need to consider multiple negative punctures). This is because the moduli space counts solutions to $(\rd u-X_H\otimes \beta)^{0,1}=0$, where we can have Floer type breaking as well as SFT type breaking at $1$. It is important in \Cref{prop:SFT1} that we ask the almost contact structure to be sufficiently stretched, for otherwise the relation could fail. However if we change $\beta$ to be $\rd t$ on $\C^*=(-\infty,+\infty)\times S^1$, and count anchored version (as in \cite{BEE}) of the following curves,
    $$\left\{ u:\CP^1\backslash \{\infty,0,1\} \to \widehat{Y}\left| (\rd u-X_H\otimes \beta)^{0,1}=0, u(\infty) = \overline{\alpha}, u(0)=\overline{\beta}, u(1) = (0,\gamma)\right.\right\}$$
    It should give rise to a map $SH^*_+(W)\to SH^*_+(W)\otimes SH^*_{+,S^1}(W)$ of degree $2n-1$, which we conjecture to be isomorphic to the secondary coproduct  $SH^*_+(W)\to SH^*_+(W)\otimes SH^*_{+,S^1}(W)$ (also of degree $2n-1$) in \cite{EGL,ekholm2017symplectic} composed with $\Id \otimes \iota $.
\end{enumerate}
\end{remark}

%% file: s5.tex
\section{Applications}\label{s5}
\begin{proof}[Proof of \Cref{thm:flexQ}]
We use $D^*Q,S^*Q$ to denote the unit disk bundle and the sphere bundle in $T^*Q$. When $\chi(Q)=0$, using the Gysin exact sequence $H^0(Q;\Q)\stackrel{\cup e(T^*Q)}{\longrightarrow} H^n(Q;\Q)\to H^n(S^*Q;\Q)$, we see that $H^n(D^*Q;\Q)\to H^n(S^*Q;\Q)$ is injective for $n=\dim Q$. Then by \Cref{prop:SS}, we have that $H^*(W;\Z)\to H^*(S^*Q;\Z)$ is isomorphic to $H^*(D^*Q;\Z)\to H^*(S^*Q;\Z)$, which is injective, for any exact filling $W$ of $\partial(\mathrm{Flex}(T^*Q))$. Then by \cite[Proposition 3.24]{filling}, we have $W$ is $D^*Q$ glued with a homology cobordism between $S^*Q$. By the same argument in \cite[\S 4]{product}, we can improve the homology cobordism to an h-cobordism when $\pi_1(Q)$ is abelian. Then the interior of $W$ is diffeomorphic to $T^*Q$ as an open manifold by the Mazur trick as in \cite[Theorem 1.2]{product}.
\end{proof}

\begin{proof}[Proof of \Cref{cor:diff}]
   By \Cref{prop:simple} and \cite[Theorem E]{filling}, any exact filling $W$ with $c_1^\Q(W)=0$ is simply connected. Since $H^*(W;\Z)=H^*(W^{flex};\Z)$ by \Cref{prop:simple} and is freely generated and supported in degree $0$ and $n$ and $\dim W\ge 6$, we have $W$ has a handle decomposition of one $0$-handle and several $n$-handles. By \cite[the remark after Corollary 4.6]{smale1962structure}, the diffeomorphism type of such manifold is uniquely determined by the intersection form under the conditions listed. Therefore $W$ is diffeomorphic to $W^{flex}$ by Theorem \ref{thm:main}. When the rank of the intersection form of $W^{flex}$ is $1$, then \Cref{prop:SS} implies that $c_1^\Q(W)=0$ automatically for any exact filling $W$. 
\end{proof}

\begin{proof}[Proof of \Cref{thm:emb}]
    In the first case, if $\partial W$ embeds into $V$ as a separating contact hypersurface with the local Liouville vector field pointing out w.r.t.\ the bounded domain $U$. Then $U$ is a symplectically aspherical filling of $\partial W$ with an exact symplectic form. Although we might not be able to find all the contact hypersurface realizing the ADC property in $U$ (see \cite[\S 8]{filling}), hence \Cref{prop:simple} does not apply directly. However, by comparing $\delta_\partial:SH^{*,<D}_+(U;\Q)\to H^{*+1}(\partial W;\Q)$ to that of $W$ for a suitable $D$, we can prove that $H^1(U;\Q)\to H^1(\partial W;\Q)$ is surjective onto the image of $H^1(W;\Q)\to H^1(\partial W;\Q)$, i.e.\ all of $H^1(\partial W;\Q)$. The surjectivity of $H^1(U;\Q)\to H^1(\partial W;\Q)$ implies that $U$ becomes a Liouville filling of $\partial W$ after a modification of the Liouville form on $V$ with a closed $1$-form on $U$. Then \Cref{prop:simple} and \Cref{thm:main} imply that the rank of the intersection form on $H_n(U;\Q)$ is the same as that of $H_n(W;\Q)$, which is larger than that of $H_n(V;\Q)$. Hence it is impossible to embed $U$ to $V$ topologically. 
    
    In the second case, if $\partial W$ embeds in $P\times \D$ as a contact surface, as $H_{2n-1}(P\times \D;\Z)=0$, it must be separating. Now if the local Liouville vector field points into the compact domain $U$ bounded by $\partial W$, by deleting $U$, we get a symplectically aspherical filling of $\partial(P\times \D)\sqcup \partial W$. Combining  \cite[Theorem 6.6]{moreno2024landscape} and \cite[Theorem 1.2.]{moreno2024rsft}, we know that $\partial(P\times \D)$ is not co-fillable, i.e.\ such a filling with two boundary components can not exist. Hence we reduce the situation back to the first case.
\end{proof}